\documentclass[a4paper,11pt]{amsart}
\usepackage[english]{babel}
\usepackage{amssymb}
\usepackage{amscd}
\usepackage{mathtools}
\usepackage[paper=a4paper,left=3cm,right=3cm,top=2cm,bottom=2.0cm]{geometry}
\usepackage{enumerate}
\usepackage{hyperref}
\usepackage{tikz}
\usepackage{tikz-cd}
\usepackage{tasks}
\usepackage{tcolorbox}

\newtheorem{thm}{Theorem}[section]
\newtheorem*{thm-non}{Theorem}

\newtheorem{lem}[thm]{Lemma}
\newtheorem{prop}[thm]{Proposition}
\newtheorem{cor}[thm]{Corollary}
\theoremstyle{definition}
\newtheorem{defi}[thm]{Definition}
\newtheorem{rem}[thm]{Remark}
\newtheorem{exam}[thm]{Example}

\DeclareMathOperator\PP{\mathbb{P}}
\DeclareMathOperator\OO{\mathcal{O}}
\DeclareMathOperator\De{D^b}

\DeclareMathOperator\Hom{Hom}
\DeclareMathOperator\Ext{Ext}

\DeclareMathOperator\dg{deg}
\DeclareMathOperator\Ev{ev}

\DeclareMathOperator\NS{NS}

\DeclareMathOperator\Co{Cone}

\DeclareMathOperator\IZ{\mathcal{I}_{\mathcal{Z}}}

\DeclareMathOperator\cok{coker}
\DeclareMathOperator\Pic{Pic}
\DeclareMathOperator\rk{rk}

\DeclareMathOperator\supp{supp}

\DeclareMathOperator{\Stab}{Stab}
\DeclareMathOperator{\SStab}{SStab}
\newcommand{\ZZ}{\ensuremath{\mathbb{Z}}} 
\newcommand{\CC}{\ensuremath{\mathbb{C}}} 
\newcommand{\X}{\ensuremath{X^{[k]}}}
\newcommand{\Xk}{\ensuremath{X^{[k]}}}

\begin{document}

\title[Stability of vector bundles on Hilbert schemes of points on K3 surfaces]{Stability of some vector bundles on Hilbert schemes of points on K3 surfaces}

\author{Fabian Reede}
\address{Institut f\"ur Algebraische Geometrie, Leibniz Universit\"at Hannover, Welfengarten 1, 30167 Hannover, Germany}
\email{reede@math.uni-hannover.de}

\author{Ziyu Zhang}
\address{Institute of Mathematical Sciences, ShanghaiTech University, 393 Middle Huaxia Road, 201210 Shanghai, P.R.China}
\email{zhangziyu@shanghaitech.edu.cn}



\begin{abstract}
Let $X$ be a projective K3 surfaces. In two examples where there exists a fine moduli space $M$ of stable vector bundles on $X$, isomorphic to a Hilbert scheme of points, we prove that the universal family $\mathcal{E}$ on $X\times M$ can be understood as a complete flat family of stable vector bundles on $M$
parametrized by $X$, which identifies $X$ with a smooth connected component of some moduli space of stable sheaves on $M$.
\end{abstract}

\maketitle

\section*{Introduction}

Let $X$ be a projective K3 surface, and $M$ a moduli space of semistable sheaves on $X$. By Mukai's seminal work \cite{muk2}, when $M$ is smooth, it is an example of the so-called irreducible holomorphic symplectic manifolds, which are an important class of building blocks in the classification of compact K\"ahler manifolds with trivial first Chern class. It is then an interesting question to understand whether the moduli spaces $\mathcal{M}$ of semistable sheaves on $M$ inherit any good properties from $M$. This paper grew out of an attempt to study this question. When $\dim M > 2$, we  cannot expect $\mathcal{M}$ to carry a holomorphic symplectic structure in general, because the Serre duality does not induce a non-degenerate anti-symmetric pairing on the tangent space of $\mathcal{M}$ any more, as opposed to the case of K3 surfaces; however, some components of $\mathcal{M}$ may nevertheless be holomorphic symplectic. 

In order to study this question, we need to classify all semistable sheaves on $M$ with fixed Chern classes, which seems difficult in general when $\dim M > 2$; it is even a challenging question to construct any non-trivial examples of semistable sheaves on $M$, due to the fact that stability is difficult to check on higher dimensional varieties in general. When $M$ is a Hilbert scheme of points on the K3 surface $X$, a natural family of vector bundles on $M$ for considering stability are the so-called tautological bundles, which were proven to be stable with respect to a suitable choice of an ample line bundle on $M$ by Schlickewei \cite{Sch10}, Wandel \cite{wandel} and Stapleton \cite{stapleton}. In fact, Wandel proved that, under some mild assumptions, the connected component of the moduli space containing the tautological bundles is isomorphic to some moduli space of vector bundles on the underlying K3 surface $X$.

There is another way to construct examples of stable sheaves on $M$. Assuming that $M$ is a fine moduli space of stable sheaves on $X$ with a universal family $\mathcal{E}$ on $X \times M$, and denoting the ``wrong-way fiber"  $\mathcal{E}|_{ \{x\} \times M}$ by $E_x$ for each closed point $x \in X$, we can ask the following questions:
\begin{itemize}
	\item Is $\mathcal{E}$ also a flat family of coherent sheaves on $M$ parametrized by $X$?
	\item If so, are the ``wrong-way'' fibers $E_x$ stable sheaves on $M$ with respect to some suitable choice of an ample line bundle for every closed point $x \in X$?
	\item If so, can we identify $X$ with a connected component of the corresponding moduli space of stable sheaves on $M$?
\end{itemize}

This idea has also been explored in the literature. In \cite{rz}, the authors studied some families of ideal sheaves and torsion sheaves of pure dimension $1$, and obtained an affirmative answer to the above questions in these cases. A systematic study of the above questions in the case of locally free sheaves was carried out in the very interesting and inspiring thesis of Wray \cite{wray}. In order to get around the difficulty of proving stability directly, he invoked the very deep and powerful technique of Hitchin-Kobayashi correspondence to translate the stability problem to the existence of some Hermitian-Einstein metrics, which was then solved by analytic methods to give affirmative answers to the above questions.

The present paper is devoted to study the above questions, in particular the stability of wrong-way fibers $E_x$ with respect to a polarization near the boundary of the ample cone of $M$, in the very classical way by showing that every proper subsheaf of $E_x$ of a smaller rank has a smaller slope. We will focus on two special cases, namely a projective K3 surface $X$ along with a Mukai vector $v$ such that either
\begin{itemize}
	\item $\NS(X) = \mathbb{Z}h$ with $h^2 = 4k$ and $v=(k+1, -h, 1)$ for any $k \geqslant 1$; or
	\item $\NS(X) = \mathbb{Z}e \oplus \mathbb{Z}f$ with the intersection matrix given by $\begin{pmatrix}
		-2k & 2k+1 \\
		2k+1 & 0
	\end{pmatrix}$ for any $k \geqslant 2$ as well as $v=(2k-1, e+(2k-1)f, 2k)$. 
\end{itemize}

We summarize our main results in the following theorem:

\begin{thm}
	For any projective K3 surface $X$ satisfying either of the above conditions,
	\begin{enumerate}[(1)]
		\item we can explicitly construct a fine moduli space $M$ of stable vector bundles of Mukai vector $v$ on $X$, isomorphic to the Hilbert scheme of $k$ points on $X$, along with a universal family $\mathcal{E}$ (see Theorem \ref{thm:1stModuli} and Theorem \ref{thm:2ndModuli});
		\item there exists an ample divisor $H$ on $M$ such that $\mathcal{E}$ can be regarded as a flat family of $\mu_H$-stable vector bundles on $M$ parametrized by $X$ (see Theorem \ref{prop:sameH1} and Theorem \ref{prop:sameH2});
		\item the classifying morphism induced by the family $\mathcal{E}$ identifies $X$ with a smooth connected component of a moduli space of $\mu_H$-stable sheaves on $M$ (see Theorem \ref{thm:component1} and Theorem \ref{thm:component2}).
	\end{enumerate}
\end{thm}

Let us briefly explain how we achieved the above results. Our choices of the K3 surfaces and the Mukai vectors, as well as the explicit constructions of the moduli space $M$ and the universal family $\mathcal{E}$ in the above two cases, are motivated by \cite[Example 5.3.7]{huy} and \cite[Theorem 1.2]{muk} respectively. In fact, in both cases, the stable sheaves on $X$ are given by the spherical twist (or its inverse) of the ideal sheaves of $k$ points on $X$ around $\OO_X$, hence their corresponding moduli spaces $M$ are isomorphic to the Hilbert scheme $X^{[k]}$ of $k$ points on $X$. To show the slope stability of the wrong-way fibers $E_x$ with respect to some ample divisor $H$ on $M$, we apply the technique developed by Stapleton \cite{stapleton}; namely, we first prove the slope stability of $E_x$ with respect to a natural nef divisor on $M$ by passing to the $k$-fold product of $X$, then use the openness of stability to perturb the nef divisor to a nearby ample divisor. In fact, since the perturbation argument in \cite{stapleton} works only for individual sheaves, we need to generalize it so as to find an ample divisor $H$ with respect to which all $E_x$'s are simultaneously stable. Finally, to identify $X$ as a smooth connected component of some moduli space of stable sheaves on $M$, we interpret $E_x$'s as images of some sheaves or derived objects on $X$ under the integral functor $\mathrm{\Phi}$ induced by the universal ideal sheaf for $X^{[k]}$. By the fundamental result of Addington \cite{adding} that $\mathrm{\Phi}$ is a $\mathbb{P}^{k-1}$-functor, we can obtain, by computing the relevant cohomology groups, that the $E_x$'s are distinct and the tangent space of deformations of each $E_x$ is of dimension $2$, which leads immediately to the conclusion.

The text is organized in three sections. The first section gives background on integral functors, while the other two deal with the two cases mentioned above respectively. All objects in this text are defined over the field of complex numbers $\mathbb{C}$. 

\subsection*{Acknowledgement}
We thank Nicolas Addington and Andrew Wray for kindly sending us \cite{wray}. We also thank Norbert Hoffmann for communicating to us Lemma \ref{quot}. We are particularly grateful to the anonymous referee who helped to improve the presentation of the manuscript greatly, and pointed out a mistake in a previous version of Proposition \ref{prop:stable}. In particular, Lemmas \ref{lem:bothagree} and \ref{lem:Hdecomp} in the current version are due to the referee.

\section{\texorpdfstring{Background on spherical twists and $\mathbb{P}^{n}$-functors }%
                               {Background on spherical twists and Pn-functors}}
Let $X$ denote a smooth projective variety with $\dim(X)=d$. As we will need them later, we quickly recall some facts about spherical twists and $\mathbb{P}^n$-functors in this section.
\newpage

\begin{defi}
An object $\mathcal{S}\in \De(X)$ is called spherical if
\begin{center}
\begin{minipage}{.7\textwidth}
\begin{enumerate}[\normalfont i)]
\item $\mathcal{S}\otimes\omega_X \cong \mathcal{S}$
\item $\Ext^i(\mathcal{S},\mathcal{S})=\begin{cases}
\mathbb{C} & \text{if}\,\,i=0,d \\ 0 & \text{otherwise}
\end{cases}$
\end{enumerate}
\end{minipage}
\end{center}
\end{defi}

\begin{rem}
We note the fact that if $X$ is a K3 surface, then any $L\in \Pic(X)$ is spherical.
\end{rem}

Using spherical objects one can construct autoequivalences of $\De(X)$ in the following way: to any object $\mathcal{F}\in \De(X)$ one can associate the following object in $\De(X\times X)$:
\begin{equation*}
\mathcal{P}_{\mathcal{F}}:=\Co(\mathcal{F}^{\vee}\boxtimes \mathcal{F}\longrightarrow \mathcal{O}_{\mathrm{\Delta}}).
\end{equation*}
We refer to \cite[\S 8]{huy3} for an exact description of the map $\mathcal{F}^{\vee}\boxtimes \mathcal{F}\rightarrow \mathcal{O}_{\mathrm{\Delta}}$ and more information.

\begin{defi}
The spherical twist associated to a spherical object $\mathcal{S}\in\De(X)$ is the Fourier-Mukai transform
\begin{equation*}
T_{\mathcal{S}}\coloneqq\mathrm{\Phi}_{\mathcal{P}_{\mathcal{S}}}: \De(X) \longrightarrow \De(X)
\end{equation*}
with kernel $\mathcal{P}_{\mathcal{S}}$.
\end{defi}

The most important fact about the spherical twist is
\begin{prop}\label{twistequi}
Let $\mathcal{S}$ be a spherical object in $\De(X)$. Then the induced spherical twist
\begin{equation*}
T_{\mathcal{S}}: \De(X) \longrightarrow \De(X)
\end{equation*}
is an autoequivalence.
\end{prop}
A proof of this proposition was given by Seidel and Thomas, see \cite[Theorem 1.2]{seidel}.

\begin{rem}\label{computetwist}
By \cite[Exercise 8.5]{huy3} the effect of the spherical twist $T_{\mathcal{S}}$ on an object $\mathcal{G}\in\De(X)$ can be described by the following distinguished triangle:
\begin{equation*}
T_{\mathcal{S}}(\mathcal{G})[-1]\longrightarrow R\Hom(\mathcal{S},\mathcal{G})\otimes \mathcal{S} \longrightarrow \mathcal{G} \longrightarrow T_{\mathcal{S}}(\mathcal{G}).
\end{equation*}
As the spherical twist $T_{\mathcal{S}}$ is an autoequivalence one can also study the inverse $T^{-1}_{\mathcal{S}}$. For any object $\mathcal{G}\in\De(X)$ there exists the following distinguished triangle, see \cite[Remark 8.11]{huy3}:
\begin{equation*}
T_{\mathcal{S}}^{-1}(\mathcal{G}) \longrightarrow \mathcal{G} \longrightarrow R\Hom(\mathcal{S},\mathcal{G})\otimes\mathcal{S}[d]\longrightarrow T_{\mathcal{S}}^{-1}(\mathcal{G})[1].
\end{equation*}
\end{rem}

We are also interested in another class of integral functors, the so-called $\mathbb{P}^n$-functors, which were introduced by Addington in a very general setting in \cite[\S 4]{adding}. We will only need the following special example: 

\begin{exam}\label{pn-func}
Let $X$ be a K3 surface, then the integral functor
\begin{equation*}\label{eqn:defPhi}
	\mathrm{\Phi} \colon \De(X) \longrightarrow \De(\X)
\end{equation*}
whose kernel is the universal ideal sheaf $\IZ$ on $X \times X^{[k]}$ is a $\PP^{k-1}$-functor with corresponding autoequivalence $H=[-2]$ by \cite[Theorem 3.1, Example 4.2(2)]{adding}. 
\end{exam}

\begin{rem}\label{pn-func2}
The fact that the above integral functor $\mathrm{\Phi}$ is a $\mathbb{P}^{k-1}$-functor with the corresponding autoequivalence $H=[-2]$ has the following useful consequence, see \cite[\S 2.1]{add16-2}: for any $E, F\in \De(X)$ we have an isomorphism of graded vector spaces
\begin{equation*}\label{eqn:pnfunc}
	\Ext^{*}_{\X}(\mathrm{\Phi}(E),\mathrm{\Phi}(F))\cong \Ext^{*}_X(E,F)\otimes H^{*}(\PP^{k-1},\CC).
\end{equation*}
\end{rem}

\section{K3 surfaces with Picard number one}\label{sect1}

Throughout this section we assume $X$ is a K3 surface such that $\NS(X) = \ZZ h$, where $h$ is an ample class with $h^2 = 4k$. We denote the line bundle associated to $h$ by $\OO_X(1)$ and the Hilbert scheme of length $k$ subschemes of $X$ by $\X$.

\subsection{Explicit construction of a universal family}

In this subsection we generalize \cite[Example 5.3.7]{huy} to give an explicit construction of a universal family of stable vector bundles on $X$ parametrized by the Hilbert scheme $X^{[k]}$ for $k\geqslant 1$. Let $h$ be the ample generator of $\NS(X)$ and $v=(k+1,-h,1)\in H^{*}_{\text{alg}}(X,\mathbb{Z})$. We have the following facts:

\begin{lem}
The moduli space $M_h(v)$ of $\mu_h$-stable sheaves on $X$ with Mukai vector $v$ is a smooth projective variety of dimension $2k$ and a fine moduli space. Furthermore every point $[E]\in M_h(v)$ represents a locally free sheaf.
\end{lem}

\begin{proof}
We note that every $\mu_h$-semistable sheaf $E$ with $v(E)=v$ is $\mu_h$-stable as $\rho(X)=1$. Thus $M_h(v)$ is a smooth projective variety. We compute:
\begin{equation*}
	\dim(M_h(v))=v^2+2=4k-2(k+1)+2=2k.
\end{equation*}
Furthermore $v'=(k+1,-h,a)$ with $a\geqslant 2$ satisfies
\begin{equation*}
	v'^2+2=4k-2a(k+1)+2\leqslant 4k-4(k+1)+2 =-2 < 0,
\end{equation*}
and thus the second Chern class is minimal (here $c_2(E)=3k$). This minimality implies that every point $[E]$ in $M_h(v)$ is given by a locally free sheaf $E$. The condition $\text{gcd}(k+1,1)=1$ implies that $M_h(v)$ is a fine moduli space  by \cite[Remark 4.6.8]{huy}.
\end{proof}

The following lemma produces examples of elements in this moduli space:
\begin{lem}\label{ideal}
	For any $[Z]\in\X$ the sheaf $I_Z(1)$ is globally generated, i.e. the evaluation morphism
	\begin{equation*}
		\Ev: H^0(I_Z(1))\otimes \OO_X \rightarrow I_Z(1)
	\end{equation*}
	is surjective. Furthermore $E_Z:=\ker(\Ev)$ is a $\mu_h$-stable locally free sheaf
	with Mukai vector given by $v(E_Z) = (k+1, -h, 1)$. 
\end{lem}

\begin{proof}
	The standard exact sequence
	\begin{equation}\label{idealseq}
		\begin{tikzcd}
			0 \arrow[r] & I_Z(1) \arrow[r] & \OO_X(1) \arrow[r] & \OO_Z(1) \arrow[r] & 0
		\end{tikzcd}
	\end{equation}	
	shows
	\begin{equation*}
		\chi(I_Z(1))=\chi(\OO_X(1))-\chi(\OO_Z(1))=(2k+2)-k=k+2.
	\end{equation*}
	
	Since $Z$ has codimension two in X, using Serre duality gives
	\begin{equation*}
		H^2(I_Z(1))\cong \Hom(I_Z(1),\OO_X)^{\vee} \cong H^0(\OO_X(-1))^{\vee}=0. 
	\end{equation*}
	
	By \cite[Proposition 3.7]{debar},
	the line bundle $\OO_X(1)$ is $k$-very ample which implies that the exact sequence of global sections attached to \eqref{idealseq}
	\begin{equation*}
		\begin{tikzcd}
			0 \arrow[r] & H^0(I_Z(1)) \arrow[r] & H^0(\OO_X(1)) \arrow[r] & H^0(\OO_Z(1)) \arrow[r] & 0
		\end{tikzcd}
	\end{equation*}	 
	is still exact. This implies $H^1(I_Z(1))\cong H^1(\OO_X(1))=0$ and thus
	\begin{equation*}
		\dim(H^0(I_Z(1)))=\chi(I_Z(1))=k+2.
	\end{equation*}
	
	Now if the evaluation map is not surjective, let $Q:=\cok(\Ev)$ and pick $x\in \supp(Q)$. Then we have an exact sequence
	\begin{equation*}
		\begin{tikzcd}
			0 \arrow[r] & I_{Z'}(1) \arrow[r] & I_Z(1) \arrow[r] & \OO_x \arrow[r] & 0
		\end{tikzcd}
	\end{equation*}
	for a length $k+1$ subscheme $Z'$ containing $Z$. 
	
	Since $I_Z(1)$ is not globally generated at $x$ the last exact sequence gives isomorphisms 
	\begin{equation*}
		H^0(I_{Z'}(1))\cong H^0(I_Z(1))\,\,\,\text{and}\,\,\,H^1(I_{Z'}(1))\cong H^0(\OO_x)\neq 0.
	\end{equation*}
	But $\OO_X(1)$ is $k$-very ample so by definition
	\begin{equation*}
		\begin{tikzcd}
			0 \arrow[r] & H^0(I_{Z'}(1)) \arrow[r] & H^0(\OO_X(1)) \arrow[r] & H^0(\OO_{Z'}(1)) \arrow[r] & 0
		\end{tikzcd}
	\end{equation*}
	is still exact, which implies $H^1(I_{Z'}(1))=0$, a contradiction. So $\Ev$ is indeed surjective and we have an exact sequence:
	\begin{equation}\label{defEZ}
		\begin{tikzcd}
			0 \arrow[r] & E_Z \arrow[r] & H^0(I_Z(1))\otimes\OO_X \arrow[r] & I_Z(1) \arrow[r] & 0.
		\end{tikzcd}
	\end{equation}
	Computing invariants shows $\rk(E_Z)=k+1$, $c_1(E_Z)=-h$ and $c_2(E_Z)=3k$, hence indeed $v(E_Z) = (k+1, -h, 1)$. The sheaf $E_Z$ is locally free as it is the kernel of a morphism between a locally free and a torsion free sheaf on a smooth surface. The stability of $E_Z$ follows from \cite[Lemma 2.1 (2-2)]{yosh}.
\end{proof}

We can globalize the construction in Lemma \ref{ideal}: let $\mathcal{Z}\subset X\times \X$ denote the universal length $k$ subscheme, $\mathcal{I}_{\mathcal{Z}}$ its ideal sheaf. There are projections $p: X\times \X \rightarrow \X$ as well as $q:X\times \X \rightarrow X$. Define a sheaf $\mathcal{E}$ on $X\times \X$ by the exact sequence
\begin{equation}\label{univ}
	\begin{tikzcd}
		0 \arrow[r] & \mathcal{E} \arrow[r] & p^{*}(p_{*}(\mathcal{I}_{\mathcal{Z}}\otimes q^{*}\OO_X(1))) \arrow[r] & \mathcal{I}_{\mathcal{Z}}\otimes q^{*}\OO_X(1) \arrow[r] & 0.
	\end{tikzcd}
\end{equation}
Then $\mathcal{E}$ is $p$-flat and $\mathcal{E}_{|p^{-1}(Z)}\cong E_Z$, which implies that $\mathcal{E}$ is locally free on $X\times \X$ by \cite[Lemma 2.1.7]{huy}. Thus $\mathcal{E}$ defines a classifying morphism
\begin{equation*}
\varphi: \X \rightarrow M_h(v),\,\, [Z]\mapsto \left[E_Z \right].
\end{equation*}

In fact we have:

\begin{thm}\label{thm:1stModuli}
The classifying morphism $\varphi: \X\rightarrow M_h(v)$ is an isomorphism.
\end{thm}
\begin{proof}
Looking at Remark \ref{computetwist} we see that the sheaf $E_Z$ defined by the exact seqeunce \eqref{defEZ} is nothing but the shifted spherical twist of $I_Z(1)$ around $\mathcal{O}_X$, more exactly we have
\begin{equation*}
E_Z=T_{\mathcal{O}_X}(I_Z(1))[1],
\end{equation*}
similar to \cite[Example 10.3.6]{huy2}. By Proposition \ref{twistequi} the spherical twist $T_{\mathcal{O}_X}$ is an autoequivalence of $\De(X)$ likewise is the shift $[1]$. But then the classifying morphism
\begin{equation*}
\varphi: \X \rightarrow M_h(v),\,\,\, [Z]\mapsto\left[ E_Z \right]=\left[T_{\mathcal{O}_X}(I_Z(1))[1] \right]  
\end{equation*}
is a composition of autoequivalences and thus maps non-isomorphic objects to non-isomorphic objects, hence $\varphi$ is injective on closed points. Since both $\X$ and $M_h(v)$ are smooth of dimension $2k$ the morphism $\varphi$ is an open embedding and thus an isomorphism as both spaces are irreducible.
\end{proof}

\subsection{Stability of wrong-way fibers}\label{stabil1}

In the above section, we explicitly constructed a universal family $\mathcal{E}$, which is a locally free sheaf on $X \times X^{[k]}$. In this section we take the alternative point of view and consider $\mathcal{E}$ as a family of vector bundles on $X^{[k]}$ parametrized by $X$. A ``wrong-way fiber" of $\mathcal{E}$ is just the restriction of $\mathcal{E}$ over a point $x\in X$ which gives a locally free sheaf on $\X$.

More precisely, we first note that by standard cohomology and base change arguments
\begin{equation*}
p_{*}(\mathcal{I}_{\mathcal{Z}}\otimes q^{*}\OO_X(1))\otimes \OO_{[Z]} \rightarrow H^0(I_Z(1))
\end{equation*}
is an isomorphism. Hence
\begin{equation}\label{eqn:K}
	K:=p_{*}(\mathcal{I}_{\mathcal{Z}}\otimes q^{*}\OO_X(1))
\end{equation}
is a locally free sheaf of rank $k+2$ on $\X$. This implies that $\mathcal{E}$ is not only $p$-flat, but also $q$-flat since $\mathcal{I}_{\mathcal{Z}}\otimes q^{*}\OO_X(1)$ is both $p$- and $q$-flat by \cite[Theorem 2.1]{krug1}. Thus we can restrict the exact sequence \eqref{univ} to the fiber over a point $x\in X$ and get the following description of the fiber $E_x:=\mathcal{E}_{|q^{-1}(x)}$:
\begin{equation}\label{defwrong}
	\begin{tikzcd}
		0 \arrow[r] & E_x \arrow[r] & K \arrow[r] & I_{S_x} \arrow[r] & 0,
	\end{tikzcd}
\end{equation}
where $S_x:=\left\lbrace [Z]\in \X\,|\, x\in \supp(Z) \right\rbrace$ is a codimension 2 subscheme of $\X$. Hence $E_x$ is a locally free sheaf of rank $k+1$ on $\X$.

Before proving the stability of $E_x$ with respect to some ample class $H \in \NS(\X)$, we recall that for any coherent sheaf $F$ on $X$ there is the associated coherent \emph{tautological sheaf} $F^{[k]}$ on $\X$ defined  by
 \begin{equation}\label{eqn:taut}
 	F^{[k]}:=p_{*}\left(q^{*}F\otimes \OO_{\mathcal{Z}} \right). 
 \end{equation}
If $F$ is locally free of rank $r$ then $F^{[k]}$ is locally free of rank $kr$.

Also recall the well-known fact that $\NS(\X)=\NS(X)_k\oplus \mathbb{Z}\delta$. Here $d_k$ is the divisor class on $\X$ induced by the divisor class $d$ on $X$ and $\delta$ is a divisor class on $\X$ such that $2\delta=[E]$ where $E$ is the exceptional divisor of the Hilbert-Chow morphism $\X\rightarrow X^{(k)}$. In our case this reads
\begin{equation*}
	\NS(\X)=\ZZ h_k\oplus\ZZ\delta.
\end{equation*}

\begin{lem}
\label{lem:Chernclass}
We have $c_1(E_x)=-h_k+\delta$.
\end{lem}

\begin{proof}
There is the exact sequence:
\begin{equation*}
	\begin{tikzcd}
		0 \arrow[r] & p_{*}(\mathcal{I}_{\mathcal{Z}}\otimes q^{*}\OO_X(1)) \arrow[r] &  p_{*}q^{*}\OO_X(1) \arrow[r] & p_{*}(\OO_{\mathcal{Z}}\otimes q^{*}\OO_X(1)) \arrow[r] & 0
	\end{tikzcd}
\end{equation*}
as $R^1p_{*}(\mathcal{I}_{\mathcal{Z}}\otimes q^{*}\OO_X(1))=0$ since $H^1(I_Z(1))=0$ for all $[Z]\in \X$.

We also have 
\begin{equation*}
p_{*}q^{*}\OO_X(1)\cong H^0(\OO_X(1))\otimes\OO_{\X}
\end{equation*}
and the sheaf $p_{*}(\OO_{\mathcal{Z}}\otimes q^{*}\OO_X(1))$ is nothing but the tautological sheaf $\OO_X(1)^{[k]}$ associated to $\OO_X(1)$ on $\X$. By \cite[Remark 3.20.]{krug} we also have $H^0(\OO_X(1)^{[k]})=H^0(\OO_X(1))$. Thus, the above exact sequence can be rewritten as
\begin{equation}
	\label{eqn:sheafK}
	\begin{tikzcd}
		0 \arrow[r] & K \arrow[r] & H^0(\OO_X(1)^{[k]})\otimes \OO_{\X} \arrow[r] & \OO_X(1)^{[k]} \arrow[r] & 0.
	\end{tikzcd}
\end{equation}

Using \cite[Lemma 1.5]{wandel} we get
\begin{equation*}
c_1(K)=-c_1(\OO_X(1)^{[k]})=-h_k+\delta.
\end{equation*}
Now exact sequence \eqref{defwrong} gives $c_1(E_x)=c_1(K)=-h_k+\delta$.
\end{proof}

To compute slopes on $\X$ we need the following intersection numbers, which can, for example, be found in \cite[Lemma 1.10]{wandel}:

\begin{lem}\label{intersect}
	For the classes $h_k$ and $\delta$ from $\NS(\X)$ we have:
	\begin{itemize}
		\setlength\itemsep{1em}
		\item $h_k^{2k}=\frac{(2k-1)!}{(k-1)!2^{k-1}}(h^2)^{k}=\frac{(2k-1)!2^{k+1}}{(k-1)!}k^k > 0$
		\item $h_k^{2k-1}\delta=0$.
	\end{itemize}
\end{lem}

We also recall the notations introduced in \cite[\S 1]{stapleton}. The ample divisor $h$ on $X$ naturally induces an ample divisor
\begin{equation*}
 h_{X^k} =\bigoplus\limits_{i=1}^k q_i^\ast h
\end{equation*}
 on $X^k$, where $q_i$ denotes the projection from $X^k$ to the $i$-th factor, as well as a semi-ample divisor $h_k$ on $\X$. 
 
Moreover, we write $X^k_\circ$, $S^kX_\circ$ and $\X_\circ$ for the loci of the relevant spaces parametrizing distinct points. Then the natural map
\begin{equation*}
\overline{\sigma}_\circ: X^k_\circ \to \X_\circ
\end{equation*}
is an \'etale cover and $j: X^k_\circ \to X^k$ is an open embedding. For any coherent sheaf $F$ on $\X$, we denote by $F_\circ$ the restriction of $F$ to $\X_\circ$, and define
\begin{equation*}
(F)_{X^k} = j_\ast (\overline{\sigma}_\circ^\ast(F_\circ))
\end{equation*}
which is a torsion free coherent sheaf if $F$ is.

\begin{prop}
	\label{prop:stableK}
	The vector bundle $K$ defined in \eqref{eqn:K} is slope stable with respect to $h_k$.
\end{prop}

\begin{proof}
	We follow the idea in the proof of \cite[Theorem 1.4]{stapleton}. 
	
	Since $(-)_\circ$ and $\overline{\sigma}_\circ^\ast(-)$ are exact, and $j_\ast(-)$ is left exact, by applying these functors to \eqref{eqn:sheafK} we obtain an exact sequence of $\mathfrak{S}_k$-invariant reflexive sheaves on $X^k$ as follows
	$$ 0 \longrightarrow (K)_{X^k} \longrightarrow (H^0(\OO_X(1))\otimes \OO_{\X})_{X^k} \stackrel{\varphi}{\longrightarrow} (\OO_X(1)^{[k]})_{X^k} $$
	where $\varphi$ is not necessarily surjective. It is clear that 
	$$ (H^0(\OO_X(1))\otimes \OO_{\X})_{X^k} = H^0(\OO_X(1)) \otimes \OO_{X^k}, $$
	and we also have
	$$ (\OO_X(1)^{[k]})_{X^k} = \bigoplus\limits_{i=1}^k q_i^\ast \OO_X(1) $$
	by \cite[Lemma 1.1]{stapleton}. Hence the above sequence becomes
	\begin{equation}
		\label{eqn:equisheaf}
		0 \longrightarrow (K)_{X^k} \longrightarrow H^0(\OO_X(1)) \otimes \OO_{X^k} \stackrel{\varphi}{\longrightarrow} \bigoplus\limits_{i=1}^k q_i^\ast \OO_X(1)
	\end{equation}
	where $\varphi$ is the evaluation map on $X^k_\circ$. 
	
	More precisely, for any set of closed points $(x_1,\ldots, x_n) \in X^k$ with $x_i \neq x_j$, the morphism of fibers can be identified as
	\begin{align*}
		\varphi_{(x_1,\ldots, x_k)}: H^0(\OO_X(1)) &\longrightarrow \bigoplus\limits_{i=1}^k \OO_X(1)_{x_i} \\
		s &\longmapsto (s(x_1),\ldots, s(x_k))
	\end{align*}
	Since for any non-trivial $s \in H^0(\OO_X(1))$, there are always (many choices of)  distinct points $(x_1,\ldots x_k) \in X^k$ such that $(s(x_1),\ldots, s(x_k)) \neq (0,\ldots, 0)$, we conclude that the map of global sections
	$$ H^0(\varphi): H^0(\OO_X(1)) \longrightarrow H^0(\bigoplus\limits_{i=1}^k q_i^\ast \OO_X(1)) $$
	is injective. It follows by exact sequence \eqref{eqn:equisheaf} that $(K)_{X^k}$ has no global sections, that is 
	\begin{equation}\label{noglobal}
	H^0((K)_{X^k}) = 0.
	\end{equation}
	
	Note that $\varphi$ is surjective on $X^k_\circ$, hence $\cok(\varphi)$ is supported on the big diagonal of $X^k$ which is of codimension $2$. It follows that
	$$ c_1((K)_{X^k}) = -\sum\limits_{i=1}^k q_i^\ast h. $$
	
	We claim that $(K)_{X^k}$ has no $\mathfrak{S}_k$-invariant subsheaf which is destabilizing with respect to $h_{X^k}$. Indeed, assume $F$ is an $\mathfrak{S}_k$-invariant subsheaf of $(K)_{X^k}$, then for some $a \in \ZZ$:
	\begin{equation*}
	c_1(F) = a(\sum\limits_{i=1}^k q_i^\ast h).
	\end{equation*}

	If $a \leqslant -1$, then 
	\begin{equation*}
	c_1(F) h_{X^k}^{2k-1} \leqslant c_1((K)_{X^k}) h_{X^k}^{2k-1} < 0
	\end{equation*}
	Since $1 \leqslant \rk(F) < \rk((K)_{X^k})$, it follows that $\mu_{h_{X^k}}(F) < \mu_{h_{X^k}}((K)_{X^k})$, hence $F$ is not destabilizing.
	
	If $a=0$, we choose a (not necessarily $\mathfrak{S}_k$-invariant) non-zero stable subsheaf $F' \subseteq F$ which has maximal slope with respect to $h_{X^k}$ (e.g. one can take a stable factor in the first Harder-Narasimhan factor of $F$). Without loss of generality, we can assume $F$ and $F'$ are both reflexive. Since $F'$ is also a subsheaf of $H^0(\OO_X(1))\otimes \OO_{X^k}$, there must be a projection from $H^0(\OO_X(1))\otimes \OO_{X^k}$ to a certain direct summand of it, such that the composition of the embedding and projection $F' \rightarrow H^0(\OO_X(1))\otimes \OO_{X^k} \rightarrow \OO_{X^k}$ is non-zero. Since $\mu_{X^k}(F') \geqslant \mu_{X^k}(F) = 0 = \mu_{X^k}(\OO_{X^k})$, and $\OO_{X^k}$ is also stable with respect to $h_{X^k}$, the map $F' \rightarrow \OO_{X^k}$ must be injective, and its cokernel is supported on a locus of codimension at least $2$. Since both are reflexive, we must have $F' = \OO_{X^k}$. Therefore $F$, and consequently $(K)_{X^k}$, have non-trivial global sections. This contradicts \eqref{noglobal}.
	
	If $a \geqslant 1$, $F$ would be a subsheaf of the trivial bundle $H^0(\OO_X(1)) \otimes \OO_{X^k}$ of positive slope. Contradiction.
	
	Finally, assume $G$ is a reflexive subsheaf of $K$. Then $(G)_{X^k}$ is an $\mathfrak{S}_k$-invariant reflexive subsheaf of $(K)_{X^k}$. By the above claim we have $\mu_{h_{X^k}}((G)_{X^k}) < \mu_{h_{X^k}}((K)_{X^k})$. It follows by \cite[Lemma 1.2]{stapleton} that $\mu_{h_k}(G) < \mu_{h_k}(K)$. Therefore $K$ is slope stable with respect to $h_k$, as desired.
\end{proof}

\begin{prop}
	\label{prop:stableEx}
	For any closed point $x\in X$, the bundle $E_x$ is slope stable with respect to $h_k$.
\end{prop}

\begin{proof}
	By Lemma \ref{lem:Chernclass}, we have $c_1(E_x) = c_1(K) = -h_k+\delta$. Therefore by Lemma \ref{intersect}
	\begin{equation*}
c_1(E_x) h_k^{2k-1} = c_1(K) h_k^{2k-1} = (-h_k+\delta) h_k^{2k-1} = -h_k^{2k} < 0.
	\end{equation*}
	Assume $F$ is a destabilizing subsheaf of $E_x$ with $1\leqslant \rk(F)\leqslant k$ and $c_1(F) = ah_k + b\delta$ for some $a,b \in \ZZ$. Then
	$$ c_1(F) h_k^{2k-1} = ah_k^{2k}. $$
	By the assumption and Proposition \ref{prop:stableK}, we have the inequality
	\begin{equation*}
\mu_{h_k}(E_x) \leqslant \mu_{h_k}(F) < \mu_{h_k}(K),
	\end{equation*}
	  which can be written as 
	  \begin{equation*}
	     \frac{-h_k^{2k}}{k+1} \leqslant \frac{ah_k^{2k}}{\rk(F)} < \frac{-h_k^{2k}}{k+2} \Longleftrightarrow -\frac{\rk(F)}{k+1} \leqslant a < -\frac{\rk(F)}{k+2}\,\,\,\text{as $h_k^{2k}>0$.}
	  \end{equation*}
	  Such an integer $a$ cannot exist. Contradiction. Hence $E_x$ is stable with respect to $h_k$.
\end{proof}

\subsection{A smooth connected component}\label{sec:pnfunc}

In this section, we will interpret the universal sheaf $\mathcal{E}$ defined in \eqref{univ} as a family of stable sheaves on $X^{[k]}$ whose base is a smooth connected component of the corresponding moduli space. We have shown above that each wrong-way fiber $E_x$ of the family $\mathcal{E}$ is $\mu_{h_k}$-stable; however, it would be more preferable to establish the stability with respect to some ample class on $X^{[k]}$. Although the perturbation technique in \cite[Proposition 4.8]{stapleton} can be used to achieve this for every single $E_x$, for our purpose we will have to extend this technique to prove that all sheaves $E_x$ are slope stable with respect to the same ample class near $h_k$.

\begin{thm}\label{prop:sameH1}
	There exists some ample class $H \in \NS(\X)$ near $h_k$, such that $E_x$ is $\mu_H$-stable for all $x \in X$ simultaneously.
\end{thm}

\begin{proof}
	Proposition \ref{prop:stableEx} and \cite[Theorem 2.3.1]{dCM} guarantees that the assumptions in \cite[Proposition 4.8]{stapleton} are satisfied for each $E_x$, hence every $E_x$ is slope stable with respect to some ample class near $h_k$ by \cite[Proposition 4.8]{stapleton}. In order to find a single ample class $H$ that is independent of the choice of $E_x$, we can literally use the entire proof of \cite[Proposition 4.8]{stapleton} except that we need to reconstruct the non-empty convex open set $U$ so that $\alpha := h_k^{2k-1}$ is in the closure of $U$, and for every $\gamma \in U$, $E_x$ is stable with respect to $\gamma$ for all $x \in X$.
	
	We follow the notations in \cite[Definition 3.1]{greb16}. For each $x \in X$, $\SStab(E_x)$ is a convex closed set containing $\alpha$. Hence the intersection 
	$$ \overline{U} := \cap_{x \in X} \SStab(E_x) $$
	is also a convex closed set containing $\alpha$. We first claim that \cite[Theorem 3.4]{greb16} holds for all $E_x$ simultaneously; namely, we will show that for any $\beta \in \mathrm{Mov}(\X)^\circ$ (see \cite[Definition 2.1]{greb16} for the notation), there exists a number $e \in \mathbb{Q}^+$, such that $(\alpha + \varepsilon \beta) \in \cap_{x \in X} \Stab(E_x)$ for any real $\varepsilon \in [0,e]$.
	
	To prove the claim, we first note that the slope $c := \mu_\beta(E_x)$ is independent of the choice of $x \in X$. We redefine the set $S$ in the proof of \cite[Theorem 3.4]{greb16} to be 
	$$ S := \{ c_1(F) \mid F \subseteq E_x \text{ for some } x \in X \text{ such that } \mu_\beta(F) \geqslant c \}. $$
	Since $E_x \subseteq K$ for all $x \in X$ by \eqref{defwrong}, we obtain that $S$ is a subset of
	$$ T := \{ c_1(F) \mid F \subseteq K \text{ such that } \mu_\beta(F) \geqslant c \}, $$
	which is finite by \cite[Theorem 2.29]{greb16}, hence $S$ is also finite. We can then use the rest of the proof of \cite[Theorem 3.4]{greb16} literally to conclude the claim.
	
	We then claim that $\overline{U}$ is of full dimension $r := \rk N_1(\X)$. If not, then we have $\alpha \in \overline{U} \subseteq L$ for some hyperplane $L \subset N_1(\X)_{\mathbb{R}}$. Since $\mathrm{Mov}(\X)$ is of full dimension, we can choose some $\beta \in \mathrm{Mov}(\X)^\circ \setminus L$. It follows that $(\alpha + \varepsilon \beta) \in \overline{U} \setminus L$ for some small $\varepsilon > 0$ by the previous claim and the choice of $\beta$. Contradiction.
	
	We define $U$ to be the interior of $\overline{U}$ and claim that $U$ is non-empty. Indeed, since $\overline{U}$ is of full dimension $r$, we can choose $r+1$ points of $\overline{U}$ in general positions, which form an $r$-simplex. By the convexity of $\overline{U}$, the entire simplex is in $\overline{U}$ hence any interior point of the simplex is also an interior point of $\overline{U}$. The convexity of $U$ follows from the convexity of $\overline{U}$. And it is clear from the construction that $\alpha = h_k^{2k-1}$ is in the closure of $U$. We finally claim that every $\gamma \in U$ is in $\cap_{x \in X} \Stab(E_x)$. If not, suppose that there exists some $\gamma_0 \in U$ and some $x_0 \in X$, such that $\gamma_0 \in \SStab(E_{x_0}) \setminus \Stab(E_{x_0})$; namely, $\mu_{\gamma_0}(F) = \mu_{\gamma_0}(E_{x_0})$ for some proper subsheaf $F$ of $E_{x_0}$. Since the slope function is linear with respect to the curve class, and $\mu_\alpha(F) < \mu_\alpha(E_{x_0})$ by Proposition \ref{prop:stableEx}, one can find a hyperplane in $N^1(\X)_\mathbb{R}$ through $\gamma_0$, such that $\mu_\gamma(E_{x_0}) - \mu_\gamma(F)$ takes opposite signs for $\gamma$ in the two open half-spaces separated by the hyperplane. In particular, $F$ destabilizes $E_{x_0}$ in one of the half-spaces. Since $U$ has non-empty intersection with both half-spaces, this contradicts the condition $U \subseteq \SStab(E_x)$. Therefore we have $U \subseteq \cap_{x\in X} \Stab(E_x)$, as desired.
\end{proof}

We give an alternative description of $E_x$ using the integral functor $\mathrm{\Phi}$ from Example \ref{pn-func}:

\begin{lem}\label{fmtrans}
	For each $x \in X$, let $I_x$ be the ideal sheaf of $x\in X$, then $E_x=\mathrm{\Phi}(I_x(1))$. 
\end{lem}

\begin{proof}
We start with the exact sequence
\begin{equation}\label{wrong2}
	\begin{tikzcd}
		0 \arrow[r] & E_x \arrow[r] & K \arrow[r] & I_{S_x} \arrow[r] & 0.
	\end{tikzcd}
\end{equation}
We note that $I_{S_x}=\mathrm{\Phi}(\OO_x)$ as $\IZ$ is flat over $X$. Furthermore we have $K=\mathrm{\Phi}(\OO_X(1))$ since $R^ip_{*}(\IZ\otimes q^{*}\OO_X(1))=0$ for $i=1,2$ as this is true for $H^i(I_Z(1))$ for any $[Z]\in \X$. These two facts imply that
\begin{equation*}
\Hom_{\X}(K,I_{S_x})=\Hom_{\X}(\mathrm{\Phi}(\OO_X(1)),\mathrm{\Phi}(\OO_x))\cong\Hom_X(\OO_X(1),\OO_x)\cong \mathbb{C}
\end{equation*}
by Remark \ref{pn-func2}. Thus the exact sequence \eqref{wrong2} is induced by the exact sequence
\begin{equation*}
	\begin{tikzcd}
		0 \arrow[r] & I_x(1) \arrow[r] & \OO_X(1) \arrow[r] & \OO_x \arrow[r] & 0.
	\end{tikzcd}
\end{equation*}
As $K\rightarrow I_{S_x}$ is surjective, applying $\mathrm{\Phi}$ to the last exact sequence shows $E_x=\mathrm{\Phi}(I_x(1))$.
\end{proof}

We return to the main result of the section. Let $H$ be an ample class that satisfies Theorem \ref{prop:sameH1}, and $\mathcal{M}$ the moduli space of $\mu_H$-stable sheaves on $\X$ with the same numerical invariants as $E_x$. Then the universal family $\mathcal{E}$ defines a classifying morphism
\begin{equation}\label{eqn:class}
f \colon X \longrightarrow \mathcal{M}, \quad x\longmapsto [E_x]
\end{equation}

In fact the morphism $f$ can be described as follows:

\begin{thm}\label{thm:component1}
	The classifying morphism \eqref{eqn:class} defined by the family $\mathcal{E}$ identifies $X$ with a smooth connected component of $\mathcal{M}$.
\end{thm}

\begin{proof}
By \cite[Lemma 1.6]{rz} we have to prove that $f$ is injective on closed points and that for all $x\in X$ we have $\dim(T_{[E_x]}\mathcal{M})=2$ .

Now by Lemma \ref{fmtrans} we know $E_x=\mathrm{\Phi}(I_x(1))$, so for $x\neq y$ we find
\begin{align*}
\Hom_{\X}(E_x,E_y)&=\Hom_{\X}(\mathrm{\Phi}(I_x(1)),\mathrm{\Phi}(I_y(1)))\\
&\cong\Hom_X(I_x(1),I_y(1))\\
&\cong\Hom_X(\OO_x,\OO_y)=0
\end{align*}
by Remark \ref{pn-func2} again. This implies $f$ is injective on closed points.

A similar computation shows
\begin{align*}
\Ext^1_{\X}(E_x,E_x)&=\Ext^1_{\X}(\mathrm{\Phi}(I_x(1)),\mathrm{\Phi}(I_x(1)))\\
&\cong\Ext^1_X(I_x(1),I_x(1))\\
&\cong\Ext^1_X(\OO_x,\OO_x)\cong T_xX.
\end{align*}
Using $T_{[E_x]}\mathcal{M}\cong \Ext^1_{\X}(E_x,E_x)$ we thus find $\dim(T_{[E_x]}\mathcal{M})=2$ as desired.
\end{proof}

\section{K3 surfaces with Picard number two}\label{sect2}

In this section, we will consider a K3 surface $X$ of Picard number $2$, and construct a complete family of stable vector bundles on the Hilbert scheme $X^{[k]}$ for $k \geqslant 2$.

\subsection{The K3 surface}

In this section we assume $X$ is a K3 surface with
\begin{equation*}
\NS(X)=\mathbb{Z}e\oplus\mathbb{Z}f
\end{equation*}
such that $e^2=-2k$, $f^2=0$ and $ef=2k+1$ for some integer $k\geqslant 2$. The existence of such K3 surfaces is guaranteed by \cite[Corollary 14.3.1]{huy2}. Since $f^2=0$, either $f$ or $-f$ is effective. Without loss of generality, we will assume that the divisor class $f$ is effective, after possibly replacing the pair $(e,f)$ by $(-e,-f)$.

In this subsection we collect some helpful properties of $X$ which will be used in the construction of some moduli spaces of stable sheaves in the next section.

\begin{lem}\label{lem:positive}
We have $D^2\geqslant 0$ for all effective divisors on $X$. Especially there are no smooth curves $C$ on $X$ with $C\cong \PP^1$.
\end{lem}
\begin{proof}
Any irreducible curve $C$ on $S$ satisfies
\begin{equation*}
C^2=C(C+K_X)=2p_a(C)-2\geqslant -2.
\end{equation*}
So assume $C^2=-2$ and write $C=me+nf$. Then we have
\begin{align*}
C^2&=(me+nf)^2=m^2e^2+2mnef\\
&=-2km^2+2(2k+1)mn\\
&=-2m(km-(2k+1)n).
\end{align*}
The equation $C^2=-2$ translates into $m(km-(2k+1)n)=1$. This implies $m=\pm 1$ but then one can see that there is no $n\in \ZZ$ satisfying this equation.
\end{proof}

\begin{lem}\label{ample}
The divisor classes $h=e+(2k-1)f$ and $\widehat{h}=(2k)e+(2k-1)f$ are ample.
\end{lem}
\begin{proof}
We have 
\begin{align*}
h^2&=(e+(2k-1)f)^2=e^2+2(2k-1)ef\\
&=-2k+2(2k-1)(2k+1)=8k^2-2k-2.
\end{align*}
So $h^2>0$ as $k\geqslant 2$. Since also $hf=ef=2k+1>0$ we see that $h$ is ample by the remark after \cite[Corollary 8.1.7]{huy2}.

A similar computation shows $\widehat{h}^2>0$ and $\widehat{h}f>0$.
\end{proof}

\begin{lem}\label{eff}
Let $m$ and $n$ be integers. If the class $me+nf$ is effective, then $0\leqslant m \leqslant \frac{2k+1}{k}n$ (thus in particular $n \geqslant 0$). Furthermore $h(me+nf)\geqslant((2k-1)(2k+1)-k)m$.
\end{lem}
\begin{proof}
Let $D$ be an effective divisor with class $me+nf$. Since the claim is additive in $m$ and $n$, we may assume w.l.o.g. that $D$ is an irreducible curve $C$.

By Lemma \ref{lem:positive} we have $C^2\geqslant 0$. Therefore
\begin{align*}
C^2&=2m\left\lbrace -km+(2k+1)n\right\rbrace  \geqslant 0 \\
hC&=(4k^2-k-1)m+\left\lbrace -km+(2k+1)n \right\rbrace > 0 
\end{align*}
which implies $m\geqslant 0$ and $-km+(2k+1)n \geqslant 0$. The last inequality can also be read as
\begin{equation*}
(2k+1)n\geqslant km \Leftrightarrow m \leqslant \frac{2k+1}{k}n.
\end{equation*}
Putting everything together shows
\begin{equation*}
0 \leqslant m \leqslant \frac{2k+1}{k}n
\end{equation*}
as well as $hC\geqslant((2k-1)(2k+1)-k)m$.
\end{proof}

\begin{cor}\label{ellip}
There is a surjective morphism $\pi: X\rightarrow \PP^1$ such that all fibers are integral curves of arithmetic genus $p_a(C)=1$, that is $X$ is elliptically fibered.
\end{cor}
\begin{proof}
Since $f^2=0$ it is known that the linear system $\left|f \right|$ induces a surjective map $\pi: X \rightarrow \PP^1$ with $\pi^{*}\OO_{\PP^1}(1)=\OO_X(f)$. By the previous lemma the class $f$ cannot be the sum of two effective divisors, hence all fibers $C$ of $\pi$ are integral and have $p_a(C)=1$. 
\end{proof}
\newpage

\begin{lem}\label{quot}
	Let $[Z]\in \Xk$. Assume $R$ is a torsion quotient of $I_Z(e)$ with $c_1(R)=nf$ for some $n\geqslant 0$, then $H^1(R)=0$.
\end{lem}
\begin{proof}
	The quotient defines the following exact sequence:
	\begin{equation*}
		\begin{tikzcd}
			0 \arrow[r] &  K \arrow[r] & I_Z(e) \arrow[r] & R \arrow[r] & 0.
		\end{tikzcd}
	\end{equation*}
	Now $K$ is torsion free of rank one, so its double dual $K^{**}$ is locally free of rank one and the natural map
	$K \to K^{**}$ is injective and the cokernel $T$ has finite support. Especially $c_1(T) = 0$ so
	\begin{equation*}
		c_1(K^{**}) = c_1(K) = c_1( I_Z(e)) - c_1( R) = e - n f
	\end{equation*}
	and thus $K^{**} \cong \mathcal{O}_X( e - n f)$. The embedding $K \hookrightarrow I_Z(e)$ induces an embedding
	\begin{equation*}
		K^{**} \cong \mathcal{O}_X( e - n f) \hookrightarrow \mathcal{O}_X(e).
	\end{equation*}
	This embedding is given by a global section of $\mathcal{O}_X( n f)$, that is by an effective divisor
	$D = \sum_i a_i C_i$ with class $n f$.
	
	This global section is the pullback along the elliptic fibration $\pi$ of a global section of
	$\mathcal{O}_{\mathbb{P}^1}( n)$, with corresponding  effective divisor $\sum_i a_i z_i$ on $\mathbb{P}^1$, here $C_i = \pi^{-1}( z_i)$.
	
	Denote by $D \subset X$ also the corresponding closed subscheme (which maybe non-reduced, if $a_i \geqslant 2$ for some $i$). We get the commutative diagram 
	\begin{equation*}
		\begin{tikzcd}
			& 0 \arrow{d} & 0 \arrow{d} & & \\
			0 \arrow{r} & K \arrow{r}\arrow{d} & \mathcal{O}_X( e - n f) \arrow{r}\arrow{d} & T \arrow{r}\arrow[d,"\alpha"] & 0\\
			0 \arrow{r} & I_Z(e) \arrow{r}\arrow{d} & \mathcal{O}_X( e) \arrow{r}\arrow{d} & \OO_Z \arrow{r} & 0\\
			& R \arrow[r,"\beta"] \arrow{d} & \mathcal{O}_D( e) \arrow{d} & &\\
			& 0  & 0  & &
		\end{tikzcd}
	\end{equation*}
	The snake lemma gives an exact sequence
	\begin{equation*}
		\begin{tikzcd}
			0 \arrow[r] & \ker( \alpha) \arrow[r] & R \arrow[r,"\beta"] & \mathcal{O}_D( e) \arrow[r] & \mathrm{coker}( \alpha) \arrow[r] & 0.
		\end{tikzcd}
	\end{equation*}
	Let $R' \subset \mathcal{O}_D( e)$ be the image of $\beta$. Since the torsion sheaf $\mathcal{O}_{\sum_i a_i z_i}$ on $\mathbb{P}^1$
	has a composition series by skyscraper sheaves $\mathcal{O}_{z_i}$ as composition factors,
	$\mathcal{O}_D$ has a composition series with composition factors $\mathcal{O}_{C_i}$,
	thus $\mathcal{O}_D( e)$ has a composition series with composition factors $\mathcal{O}_{C_i}( e)$.
	The latter is a line bundle of degree
	\begin{equation*}
		e\cdot C_i = e f = 2k+1
	\end{equation*}
	on $C_i$. The quotient $\mathcal{O}_D( e)/R'$ is isomorphic to $\mathrm{coker}(\alpha)$, that is to a quotient $Q$ of $\OO_Z$.
	By intersecting with $R'$ we get a composition series for $R'$ with composition factors which are kernels of a surjection $\mathcal{O}_{C_i}( e) \twoheadrightarrow Q'$ with $Q'$ of length $\leqslant k$. Thus we have exact sequences:
	\begin{equation*}
		\begin{tikzcd}
			0 \arrow[r] & L \arrow[r] & \OO_{C_i}(e) \arrow[r] & Q' \arrow[r] & 0,
		\end{tikzcd}
	\end{equation*}
	with a torsion free sheaf $L$ of rank one on the integral projective curve $C_i$ of arithmetic genus one. Using $\chi(\OO_{C_i})=0$ and  
	 \begin{equation*}
	 	\chi(L)=\chi(\OO_{C_i}(e))-\chi(Q')\geqslant k+1, 
	 \end{equation*}
	gives 
	\begin{equation*}
		\dg(\OO_{C_i}(e))\geqslant \dg(L)\geqslant k+1.
	\end{equation*}
	By \cite[Proposition 4.6.]{sou} all of these composition factors have trivial $H^1$. By constructing short exact sequences out of the composition series and using the induced exact sequences for $H^1$, it follows
	\begin{equation*}
		H^1(R') = 0.
	\end{equation*}
	As $\ker( \beta) = \ker( \alpha) \subseteq T$ has finite support, we also have $H^1( \ker( \beta)) = 0$. Hence
	\begin{equation*}
		H^1(R) = 0.\qedhere
	\end{equation*}
\end{proof}

\subsection{The construction of a universal family}

In this subsection we want to generalize \cite[Theorem 1.2]{muk}. Let $h$ be the ample line bundle defined in Lemma \ref{ample}, and for any integer $k \geqslant 2$ let $v=(2k-1,h,2k) \in H^\ast_{\mathrm{alg}}(X, \mathbb{Z})$. We immediately have the following result:

\begin{lem}\label{lem:standardstuff}
The moduli space $M_h(v)$ of $\mu_h$-stable sheaves on $X$ with Mukai vector $v$ is a smooth projective variety of dimension $2k$ and a fine moduli space, and every point $[E] \in M_h(v)$ represents a locally free sheaf.	
\end{lem}

\begin{proof}
We first observe by \cite[Lemma 1.2.7]{huy} that every $\mu_h$-semistable sheaf $E$ with Mukai vector $v(E)=v$ is $\mu_h$-stable since $\gcd(2k-1,h^2)=1$. Thus $M_h(v)$ is a smooth projective variety. We compute:
\begin{equation*}
	\dim(M_h(v))=v^2+2=(8k^2-2k-2)-2(2k-1)(2k)+2=2k.
\end{equation*}
Furthermore $v'=(2k-1,h,a)$ with $a\geqslant 2k+1$ satisfies
\begin{equation*}
	v'^2+2=h^2-2a(2k-1)+2\leqslant(8k^2-2k-2)-2(2k-1)(2k+1)+2= 2-2k < 0,
\end{equation*}
so again every point $[E]$ in $M_h(v)$ is locally free. The condition $\gcd(2k-1,h^2)=1$ also implies that $M_h(v)$ is a fine moduli space.
\end{proof}

\medskip

In the following discussion, we will explicitly construct a universal family for the moduli space $M_h(v)$. We first define some integral functors. For any line bundle $L$ on $X$, we define
$$ M_L: \De(X) \longrightarrow \De(X); \quad (-) \longmapsto (-) \otimes L. $$
Then we consider the composition
\begin{equation}\label{eqn:Theta}
	\mathrm{\Theta} \coloneqq M_{\mathcal{O}_X(f)} \circ T_{\mathcal{O}_X}^{-1} \circ M_{\mathcal{O}_X(e)}: \De(X) \longrightarrow \De(X),
\end{equation}
where $T_{\mathcal{O}_X}^{-1}$ is the inverse of the spherical twist induced by $\mathcal{O}_X$. It is clear that $\mathrm{\Theta}$ is an autoequivalence of $\De(X)$ hence a Fourier-Mukai transform. We denote the corresponding kernel by $\mathcal{P} \in \De(X \times X)$. By Remark \ref{computetwist}, we have an explicit description of $\mathcal{P}$ by the exact triangle
\begin{equation}\label{eqn:Ptriangle}
	\mathcal{P} \longrightarrow \mathrm{\Delta}_\ast \OO_X(e+f) \longrightarrow \OO_X(e) \boxtimes \OO_X(f)[2] \longrightarrow \mathcal{P}[1],
\end{equation}
where $\mathrm{\Delta}: X \hookrightarrow X \times X$ is the diagonal embedding. The kernel $\mathcal{P}$ also defines a Fourier-Mukai transform in the opposite direction, which we denote by 
$$ \widehat{\mathrm{\Theta}}: \De(X) \longrightarrow \De(X). $$
Since the kernel of each composition factor in \eqref{eqn:Theta}, viewed as an object in $\De(X \times X)$, remains the same under the permutation of the two copies of $X$, it follows that
\begin{equation}\label{eqn:Thetahat}
	\widehat{\mathrm{\Theta}} = M_{\mathcal{O}_X(e)} \circ T_{\mathcal{O}_X}^{-1} \circ M_{\mathcal{O}_X(f)}.
\end{equation}
For any $[Z] \in X^{[k]}$, we apply $\mathrm{\Theta}$ on the ideal sheaf $I_Z$ and define
\begin{equation*}
	E_Z \coloneqq \mathrm{\Theta}(I_Z).
\end{equation*}
A priori $E_Z$ is a derived object on $X$, but we can show the following:

\begin{thm}\label{thm:2ndModuli}
	$E_Z$ is $\mu_h$-stable locally free sheaf with Mukai vector $v(E_Z) = (2k-1, h, 2k)$.
\end{thm}

\begin{proof}
First of all, by \eqref{eqn:Theta} and \cite[Lemma 8.12]{huy3}, a standard computation of the cohomological Fourier-Mukai transform shows that $$v(E_Z) = (2k-1, h, 2k).$$ Moreover, by \eqref{eqn:Ptriangle} and the fact that $T^{-1}_{\OO_X}(\OO_X) = \OO_X[1]$, we obtain an exact triangle
\begin{equation}\label{eqn:EZtriangle}
	E_Z \longrightarrow I_Z(e+f) \longrightarrow H^\ast(I_Z(e)) \otimes \OO_X(f)[2] \longrightarrow E_Z[1].
\end{equation}
In order to compute $H^\ast(I_Z(e))$, we observe by Lemma \ref{eff} that
\begin{equation}\label{eqn:vanishinge}
	h^0(\OO_X(e)) = 0 \quad \text{and} \quad h^2(\OO_X(e)) = h^0(\OO_X(-e)) = 0.
\end{equation}
It follows by a long exact sequence of cohomology groups that
$$ h^0(I_Z(e)) = h^2(I_Z(e)) = 0. $$
Therefore the exact triangle \eqref{eqn:EZtriangle} reduces to the short exact sequence
\begin{equation}\label{exseq}
	0 \longrightarrow H^1(I_Z(e)) \otimes \OO_X(f) \longrightarrow E_Z \longrightarrow I_Z(e+f) \longrightarrow 0,
\end{equation}
where $\dim H^1(I_Z(e)) = \rk(E_Z)-1 = 2k-2$. For the convenience of analyzing the stability of $E_Z$, we rewrite the above exact triangle as
\begin{equation*}\label{exseq2}
	0 \longrightarrow \OO_X^{\oplus (2k-2)} \longrightarrow E_Z(-f) \longrightarrow I_Z(e) \longrightarrow 0.
\end{equation*}
Furthermore, we observe that $\OO_X(f) = \mathrm{\Theta}(\OO_X(-e)[-1])$. Since $\mathrm{\Theta}$ is an equivalence, we have
$$ \Hom(E_Z(-f), \OO_X) = \Hom(E_Z, \OO_X(f)) = \Hom(I_Z, \OO_X(-e)[-1]) = 0. $$

We are now ready to prove that $E_Z$, or rather $E_Z(-f)$, is $\mu_h$-stable. We first have
	\begin{equation*}
		\mu_h(E_Z(-f))=\frac{eh}{2k-1}=\frac{-2k+(2k-1)(2k+1)}{2k-1}=2k+1-\frac{2k}{2k-1}>0.
	\end{equation*}

	\medskip	
	Pick a torsion free quotient  $F$ of $E_Z(-f)$ with $1\leqslant\rk(F)\leqslant 2k-2$. We have
	\begin{equation*}
		\begin{tikzcd}
			E_Z(-f) \arrow{r} & F \arrow{r} & 0
		\end{tikzcd}
	\end{equation*}
	with $\Hom(F,\mathcal{O}_X)\hookrightarrow \Hom(E_Z(-f),\mathcal{O}_X)=0$.
	
	We want to show that we always have $\mu_h(F)>\mu_h(E_Z(-f))$. For this, define the torsion free sheaf $F_0$ as the image of the composition
	\begin{equation*}
		\begin{tikzcd}
			\mathcal{O}_X^{\oplus (2k-2)} \arrow[hookrightarrow]{r} & E_Z(-f) \arrow{r} & F.
		\end{tikzcd}
	\end{equation*}
	We get a surjection
	\begin{equation*}
		\begin{tikzcd}
			\mathcal{O}_X^{\oplus (2k-2)} \arrow{r} & F_0 \arrow{r} & 0.
		\end{tikzcd}
	\end{equation*}
	This implies that $c_1(F_0)$ is effective and we have the following commutative diagram:
	\begin{equation}\label{big}
		\begin{tikzcd}
			& 0 \arrow{d} & 0 \arrow{d} & 0 \arrow{d} & \\
			0 \arrow{r} & K_0 \arrow{r}\arrow{d} & K_1 \arrow{r}\arrow{d} & K_2 \arrow{r}\arrow{d} & 0\\
			0 \arrow{r} & \mathcal{O}_X^{\oplus (2k-2)} \arrow{r}\arrow{d} & E_Z(-f) \arrow{r}\arrow{d} & I_Z(e) \arrow{r}\arrow{d} & 0\\
			0 \arrow{r} & F_0 \arrow{r}\arrow{d} & F \arrow{r}\arrow{d} & F_1 \arrow{r}\arrow{d} & 0\\
			& 0  & 0  & 0  &
		\end{tikzcd}
	\end{equation}
	
	\medskip
	Due to the diagram $\rk(F_1)\in\left\lbrace 0,1\right\rbrace$.
	\medskip
	
	Case 1: $\rk(F_1)=1$. Then $\rk(F_0)=\rk(F)-1$ and $F_1\cong I_Z(e)$. We conclude
	\begin{equation*}
		c_1(F)=c_1(F_0)+c_1(I_Z(e)) \Rightarrow c_1(F)=c_1(F_0)+e.
	\end{equation*}
	
	Using this we find:
	\begin{equation*}
		\mu_h(F)=\frac{c_1(F)h}{\rk(F)}=\underbrace{\frac{c_1(F_0)h}{\rk(F)}}_{\geqslant 0}+\frac{eh}{\rk(F)}>\frac{eh}{2k-1}=\mu_h(E_Z(-f)).
	\end{equation*}
	
	So we indeed have $\mu_h(F)>\mu_h(E_x(-f))$.
	\medskip
	
	Case 2: $\rk(F_1)=0$. Now $\rk(F_0)=\rk(F)$. Write $c_1(F)=me+nf$. Since $c_1(F_0)$ and $c_1(F_1)$ are effective, so is their sum $c_1( F)$, which by Lemma \ref{eff} implies, that $m\geqslant 0$ as well as 
	\begin{equation*}
		\mu_h(F)=\frac{(me+nf)h}{\rk(F)}\geqslant\frac{m((2k-1)(2k+1)-k)}{\rk(F)}\geqslant m(2k+1-\frac{k}{2k-1}).
	\end{equation*}
	For $m\geqslant 1$ we have
	\begin{align*}
		\mu_h(F) &\geqslant m(2k+1-\frac{k}{2k-1})\\
		&\geqslant 2k+1-\frac{k}{2k-1}\\ 
		&> 2k+1-\frac{2k}{2k-1}=\mu_h(E_Z(-f)) 
	\end{align*}
	\medskip
	
	So only the case $m=0$ remains, i.e. $c_1(F)=nf$. We have
	\begin{equation*}
		\mu_h(F)=\frac{n(2k+1)}{\rk(F)}.
	\end{equation*}
	If we can prove $n\geqslant \rk(F)$ we are done since then
	\begin{equation*}
		\mu_h(F)\geqslant 2k+1 > 2k+1-\frac{2k}{2k-1} =\mu_h(E_Z(-f)).
	\end{equation*}
	
	As $c_1(F)=nf$ is the sum of the two effective divisors $c_1(F_0)$ and $c_1(F_1)$, it follows from Lemma \ref{eff} that $c_1( F_0) = n_0 f$ and $c_1( F_1) = n_1f$ with $n_0, n_1 \geqslant 0$ and $n_0 + n_1 = n$.
	
	By Lemma \ref{quot} we have $H^1(F_1) = 0$ which implies $\Ext^1( F_1, \mathcal{O}_X) = 0$ using Serre duality. So the restriction map
	\begin{equation*}
		\Hom( F, \mathcal{O}_X) \to \Hom( F_0, \mathcal{O}_X)
	\end{equation*}
	surjective. But we know $\Hom( F, \mathcal{O}_X) = 0$. So
	\begin{equation} \label{eq:F_0}
		\Hom( F_0, \mathcal{O}_X) = 0.
	\end{equation}
	
	Using the elliptic fibration $\pi:X\rightarrow \PP^1$ we have:
	\begin{equation}\label{det}
		h^0(\det(F_0))=h^0(\mathcal{O}_X(n_0f))=n_0+1.
	\end{equation}
   Now there is a trivial sub-bundle in $\mathcal{O}_X^{\oplus (2k-2)}$ of rank $\rk(F)+1$ such that
	\begin{equation*}
		\begin{tikzcd}
			\mathcal{O}_X^{\oplus (\rk(F)+1)} \arrow[r,"\varphi"] & F_0 
		\end{tikzcd}
	\end{equation*}
	is surjective outside a finite subset of $X$ by \cite[Lemma 4.60]{bar}.
	\medskip
	
	Define $R:=\cok(\varphi)$. Then there is the exact sequence:
	\begin{equation*}
		\begin{tikzcd}
			0 \arrow{r} & F_0' \arrow{r} & F_0 \arrow{r} & R \arrow{r} & 0.
		\end{tikzcd}
	\end{equation*}
	As $R$ has finite support, we get: 
	\begin{center}
		$\det(F_0)=\det(F_0')$ as well as $H^2(F_0')\cong H^2(F_0)$.
	\end{center}
	
	We also have the exact sequence
	\begin{equation*}
		\begin{tikzcd}
			0 \arrow{r} & \det(F_0)^{-1} \arrow{r} & \mathcal{O}_X^{\oplus(\rk(F)+1)} \arrow{r} & F_0' \arrow{r}  & 0.
		\end{tikzcd}
	\end{equation*}
	The end of the induced long cohomology sequence gives:
	\begin{equation}\label{les}
		\begin{tikzcd}
			H^1(F_0') \arrow{r} & H^2(\det(F_0)^{-1}) \arrow{r} & H^2(\mathcal{O}_X^{\oplus (\rk(F)+1)}) \arrow{r} & H^2(F_0') \arrow{r}  & 0.
		\end{tikzcd}
	\end{equation}
	
	It follows from \eqref{eq:F_0} by Serre duality that
	\begin{equation*}
		H^2(F_0')\cong H^2(F_0)\cong \Hom(F_0,\mathcal{O}_X)^\vee =0.
	\end{equation*}
	
	Since $H^2(F_0')=0$, we apply Serre duality again and obtain from \eqref{les} that
	\begin{equation*}
		\begin{tikzcd}
			0 \arrow{r} & H^0(\mathcal{O}_X^{\oplus (\rk(F)+1)}) \arrow{r} & H^0(\det(F_0)).
		\end{tikzcd}
	\end{equation*}
	We conclude
	\begin{equation*}
		h^0(\det(F_0))\geqslant \rk(F)+1.
	\end{equation*}
	Using this inequality together with \eqref{det} we get:
	\begin{equation*}
		n_0+1=h^0(\det(F_0))\geqslant \rk(F)+1 \Rightarrow n_0 \geqslant \rk(F) \Rightarrow n \geqslant \rk(F).
	\end{equation*}
	We obtain the desired inequality between $n$ and $\rk(F)$, hence $E_Z(-f)$ is stable, and so is $E_Z$. It then follows by Lemma \ref{lem:standardstuff} that $E_Z$ is locally free.
\end{proof}

We want to globalize the previous construction. For this we denote the universal closed subscheme of length $n$ by $\mathcal{Z}\subset X\times \Xk$, and the universal ideal sheaf by $\mathcal{I}_\mathcal{Z}$. As a kernel, $\mathcal{I}_\mathcal{Z}$ induces a pair of integral functors (in opposite directions):
\begin{equation*}
	\mathrm{\Phi}: \De(X) \longrightarrow \De(X^{[k]}) \,\,\,\,\text{and}\,\,\,\,\widehat{\mathrm{\Phi}}: \De(X^{[k]}) \longrightarrow \De(X).
\end{equation*}
Here $\mathrm{\Phi}$ is a $\mathbb{P}^{k-1}$-functor, see Example \ref{pn-func}.

The composition of the integral functors
$$ \mathrm{\Theta} \circ \widehat{\mathrm{\Phi}}: \De(X^{[k]}) \longrightarrow \De(X) $$
is still an integral functor, whose kernel $\mathcal{E} \in \De(X^{[k]} \times X)$ can be computed from $\mathcal{P}$ and $\mathcal{I}_\mathcal{Z}$ explicitly. More precisely, let $\pi_{12}$, $\pi_{23}$ and $\pi_{13}$ be projections from $X^{[k]} \times X \times X$ to each pair of factors, then
$$ \mathcal{E} = R{\pi_{13}}_\ast(\pi_{12}^\ast\mathcal{I}_\mathcal{Z}\otimes\pi_{23}^\ast\mathcal{P}); $$
see \cite[Proposition 5.10]{huy3}. We have the following property about $\mathcal{E}$:

\begin{prop}\label{prop:universal-family}
	$\mathcal{E}$ is a locally free sheaf on $X^{[k]} \times X$ such that $\mathcal{E}|_{\{[Z]\} \times X} \cong E_Z$ for any $[Z] \in X^{[k]}$.
\end{prop}

\begin{proof}
	For any $[Z] \in X^{[k]}$, the derived pullback of $\mathcal{E}$ to the fiber $\{[Z]\} \times X$ can be computed by
	$$ (\mathrm{\Theta} \circ \widehat{\mathrm{\Phi}}) (\mathcal{O}_{[Z]}) \cong \mathrm{\Theta}(I_Z) = E_Z, $$
	which is a locally free sheaf by Theorem \ref{thm:2ndModuli}. It follows that $\mathcal{E}$ is a sheaf which is flat over $X^{[k]}$ by \cite[Lemma 3.31]{huy3}, and locally free by \cite[Lemma 2.1.7]{huy}.
\end{proof}

In fact, $\mathcal{E}$ is a universal family for the fine moduli space $M_h(v)$. More precisely, we have

\begin{cor}
	The family $\mathcal{E}$ induces an isomorphism $\Xk \cong M_h(v)$.
\end{cor}
\begin{proof}
	$\mathcal{E}$ induces a classifying morphism
	\begin{equation*}
		\varphi: \Xk \longrightarrow M_h(v); \quad [Z] \longmapsto \left[ E_Z \right].
	\end{equation*}
	Since $\mathrm{\Theta}$ is an equivalence, we have $E_Z \not\cong E_{Z'}$ for $[Z] \neq [Z']$, hence $\varphi$ is injective, hence it is an open embedding since $\Xk$ and $M_h(v)$ are both of dimension $2k$. But $\Xk$ is projective, so $\varphi$ is also closed. Since $\Xk$ and $M_h(v)$ are both irreducible, $\varphi$ must be an isomorphism.
\end{proof}

\begin{rem}\label{rmk:correctfiber}
	Although it is not strictly required in our following discussion, the universal family $\mathcal{E}$ can in fact be given in a more explicit form similar to \eqref{exseq}. To globalize the construction in Theorem \ref{thm:2ndModuli}, we apply the functor $R{\pi_{13}}_\ast(\pi_{12}^\ast \mathcal{I}_\mathcal{Z} \otimes \pi_{23}^\ast (-))$ to \eqref{eqn:Ptriangle} and obtain
	$$ \mathcal{E} \longrightarrow R{\pi_{13}}_\ast(\pi_{12}^\ast \mathcal{I}_\mathcal{Z} \otimes \pi_{23}^\ast \mathrm{\Delta}_\ast\OO_X(e+f)) \longrightarrow R{\pi_{13}}_\ast(\pi_{12}^\ast \mathcal{I}_\mathcal{Z} \otimes \pi_2^\ast \OO_X(e) \otimes \pi_3^\ast \OO_X(f))[2] \longrightarrow \mathcal{E}[1]. $$
	We denote the projections from $X^{[k]} \times X$ to the two factors by $p$ and $q$ respectively. Then a simple calculation reduces the above exact triangle to
	$$ \mathcal{E} \longrightarrow \mathcal{I}_\mathcal{Z} \otimes q^\ast\OO_X(e+f) \longrightarrow Rp_\ast(\mathcal{I}_\mathcal{Z} \otimes q^\ast\OO_X(e)) \boxtimes \OO_X(f)[2] \longrightarrow \mathcal{E}[1]. $$
	For the consistency with the following discussion, we denote
	$$ \mathcal{H} \coloneqq Rp_\ast(\mathcal{I}_\mathcal{Z} \otimes q^\ast\OO_X(e))[1] = \mathrm{\Phi}(\OO_X(e))[1]. $$
	We will prove in Lemma \ref{lem:Exextension} that $\mathcal{H}$ is in fact a sheaf. Therefore the exact triangle reduces to
	$$ 0 \longrightarrow \mathcal{H} \boxtimes \OO_X(f) \longrightarrow \mathcal{E} \longrightarrow \mathcal{I}_{\mathcal{Z}}\otimes q^{*}\OO_{X}(e+f) \longrightarrow 0. $$
\end{rem}

\subsection{The wrong-way fibers}

In this subsection we study the wrong-way fibers of $\mathcal{E}$. For any $x \in X$, we define the corresponding wrong-way fiber to be
$$ E_x \coloneqq \mathcal{E}|_{X^{[k]} \times \{x\}}, $$
which is locally free of rank $2k-1$. As an alternative description, we consider the composition
$$ \mathrm{\Phi} \circ \widehat{\mathrm{\Theta}}: \De(X) \longrightarrow \De(X^{[k]}), $$
which is also an integral functor with kernel $\mathcal{E}$, in the direction opposite to $\mathrm{\Theta} \circ \widehat{\mathrm{\Phi}}$. Then we have
$$ E_x = (\mathrm{\Phi} \circ \widehat{\mathrm{\Theta}})(\mathcal{O}_x). $$
The following result gives a concrete description of $E_x$:

\begin{lem}\label{lem:Exextension}
	The locally free sheaf $E_x$ fits in an exact sequence of sheaves
	\begin{equation}\label{defwrong2}
		\begin{tikzcd}
			0 \arrow[r] & \mathcal{H} \arrow[r] & E_x \arrow[r] & I_{S_x} \arrow[r] & 0,
		\end{tikzcd}
	\end{equation}
	where 
	$$\mathcal{H} \coloneqq \mathrm{\Phi}(\OO_X(e))[1]$$
	is locally free, and $I_{S_x}$ is the ideal sheaf of
	$$S_x \coloneqq \left\lbrace [Z]\in \Xk \, \middle| \, x\in \supp(Z) \right\rbrace \subset X^{[k]}.$$
\end{lem}

\begin{proof}
	We write $F_x \coloneqq \widehat{\mathrm{\Theta}}(\mathcal{O}_x)$, then $E_x = \mathrm{\Phi}(F_x)$. By \eqref{eqn:Thetahat} we have $F_x = T^{-1}_{\mathcal{O}_X}(\mathcal{O}_x) \otimes \mathcal{O}_X(e)$. By applying the inverse spherical functor $T_{\mathcal{O}_X}^{-1}$ on the exact sequence
	$$ 0 \longrightarrow I_x \longrightarrow \OO_X \longrightarrow \OO_x \longrightarrow 0 $$
	we obtain an exact triangle
	$$ T^{-1}_{\mathcal{O}_X}(I_x) \longrightarrow T^{-1}_{\mathcal{O}_X}(\OO_X) \longrightarrow T^{-1}_{\mathcal{O}_X}(\OO_x) \longrightarrow T^{-1}_{\mathcal{O}_X}(I_x)[1]. $$
	Since $T_{\mathcal{O}_X}(\OO_X) = \OO_X[-1]$ and $T_{\OO_X}(\OO_x) = I_x[1]$, the above triangle becomes
	$$ \OO_x[-1] \longrightarrow \OO_X(e)[1] \longrightarrow F_x \longrightarrow \OO_x. $$
	Since $\mathrm{\Phi}(\OO_x) = I_{S_x}$, we further apply the integral functor $\mathrm{\Phi}$ to obtain the exact triangle
	\begin{equation}\label{eqn:Pre1st}
		I_{S_x}[-1] \longrightarrow \mathcal{H} \longrightarrow E_x \longrightarrow I_{S_x},
	\end{equation}
	where $\mathcal{H} = \mathrm{\Phi}(\OO_X(e))[1]$. To compute $\mathcal{H}$, we observe that the short exact sequence of kernels
	$$ 0 \longrightarrow \mathcal{I}_\mathcal{Z} \longrightarrow \OO_{X^{[k]} \times X} \longrightarrow \OO_\mathcal{Z} \longrightarrow 0 $$
	induces an exact triangle
	\begin{equation}\label{eqn:Pre2nd}
		\mathrm{\Phi}(\OO_X(e)) \longrightarrow H^\ast(\OO_X(e)) \otimes \OO_{X^{[k]}} \longrightarrow \OO_X(e)^{[k]} \longrightarrow \mathrm{\Phi}(\OO_X(e))[1].
	\end{equation}
	Since $H^i(\OO_X(e)) = 0$ for $i \neq 1$ by \eqref{eqn:vanishinge}, the exact triangle \eqref{eqn:Pre2nd} reduces to the short exact sequence
	$$ 0 \longrightarrow \OO_{X}(e)^{[k]} \longrightarrow \mathcal{H} \longrightarrow H^1(\OO_X(e))\otimes \OO_{\Xk} \longrightarrow 0, $$
	which in particular implies that $\mathcal{H}$ is a locally free sheaf. It follows that the exact triangle \eqref{eqn:Pre1st} reduces to the short exact sequence \eqref{defwrong2}.
\end{proof}

We will require a technical result in the proof of the stability. For this purpose, we define
$$ I^k X \coloneqq (X^k \times_{S^k X} X^{[k]})_\textrm{red} $$
to be Haiman's isospectral Hilbert scheme, and denote its projections to both factors by $p$ and $q$ respectively. Then the derived McKay correspondence
$$ \mathrm{\Psi} \coloneqq (-)^{\mathfrak{S}_k} \circ q_\ast \circ Lp^\ast: \De(X^k)^{\mathfrak{S}_k} \longrightarrow \De(X^{[k]}) $$
is an equivalence, and so is $\mathrm{\Psi}^{-1}: \De(X^{[k]}) \to \De(X^k)^{\mathfrak{S}_k}$. We have

\begin{lem}\label{lem:bothagree}
	For any coherent sheaf $F$ on $X^{[k]}$, if $\mathrm{\Psi}^{-1}(F)$ is a reflexive sheaf, then $$ \mathrm{\Psi}^{-1}(F) = (F)_{X^k}. $$
\end{lem}

\begin{proof}
	We follow the above notation to denote
	$$ I^k X_\circ \coloneqq X_\circ^k \times_{S^k X_\circ} X_\circ^{[k]}, $$
	then we have the commutative diagram
	\begin{equation*}
		\begin{tikzcd}
			X_\circ^{[k]} \arrow[d,"\alpha"] & I^kX_\circ \arrow[r,"p_\circ"] \arrow[d,"\beta"] \arrow[l,"q_\circ"'] & X_\circ^k \arrow[d,"j"] \\
			X^{[k]} & I^kX \arrow[l,"q"'] \arrow[r,"p"] & X^k,
		\end{tikzcd}
	\end{equation*}
	where $\alpha$, $\beta$, $j$ and $q_\circ$ are \'etale morphisms, and $p_\circ$ is an isomorphism. We also have
	\begin{align*}
		\mathrm{\Psi}^{-1} &\cong Rp_\ast \circ q^!, \\
		(-)_{X^k} &= j_\ast \circ \overline{\sigma}_\circ^\ast \circ \alpha^\ast,
	\end{align*}
	where the first equation is due to the fact that $\mathrm{\Psi}^{-1}$ is the right adjoint of $\mathrm{\Psi}$. It follows that
	\begin{align*}
		j^\ast \circ \mathrm{\Psi}^{-1} &\cong j^\ast \circ Rp_\ast \circ q^! \cong {p_\circ}_\ast \circ \beta^\ast \circ q^! \\
		&\cong {p_\circ}_\ast \circ \beta^! \circ q^! \cong {p_\circ}_\ast \circ q_\circ^! \circ \alpha^! \\
		&\cong {p_\circ}_\ast \circ q_\circ^\ast \circ \alpha^\ast \cong \overline{\sigma}_\circ^\ast \circ \alpha^\ast.
	\end{align*}
	Therefore we have
	$$ j_\ast \circ j^\ast \circ \mathrm{\Psi}^{-1} \cong (-)_{X^k}. $$
	Since $\mathrm{\Delta} = X^k \setminus X^k_\circ$ is of codimension $2$, if $\mathrm{\Psi}^{-1}(F)$ is a reflexive sheaf, then we have
	$$ \mathrm{\Psi}^{-1}(F) \cong j_\ast \circ j^\ast \circ \mathrm{\Psi}^{-1}(F) \cong (F)_{X^k} $$
	as desired.
\end{proof}

\begin{lem}\label{lem:Hdecomp}
	The sheaf $(\mathcal{H})_{X^k}$ fits in an exact sequence of $\mathfrak{S}_k$-invariant locally free sheaves
	\begin{equation}\label{eqn:Hdecomp}
		0 \longrightarrow \bigoplus_{i=1}^k q_i^\ast \OO_X(e) \longrightarrow (\mathcal{H})_{X^k} \longrightarrow H^1(\OO_X(e)) \otimes \OO_{X^k} \longrightarrow 0.
	\end{equation}
	Moreover, every $\mathfrak{S}_k$-invariant global section of $(\mathcal{H})_{X^k}$ vanishes; namely $H^0((\mathcal{H})_{X^k})^{\mathfrak{S}_k} = 0$.
\end{lem}

\begin{proof}
	By \cite[Theorem 3.6]{krug18}, the composition $\mathrm{\Psi}^{-1} \circ \mathrm{\Phi}: \De(X) \to \De(X^n)^{\mathfrak{S}_k}$ agrees with the truncated universal ideal functor defined in \cite[Definition 5.1]{krugsosna}, therefore we have an exact triangle
	\begin{equation}\label{eqn:finaltriangle}
		(\mathrm{\Psi}^{-1} \circ \mathrm{\Phi})(\OO_X(e)) \longrightarrow H^\ast(\OO_X(e)) \otimes \OO_{X^k} \stackrel{\delta}{\longrightarrow} \bigoplus_{i=1}^k q_i^\ast \OO_X(e) \longrightarrow (\mathrm{\Psi}^{-1} \circ \mathrm{\Phi})(\OO_X(e))[1],
	\end{equation}
	where each component of $\delta$ is an evaluation map. Since $\mathcal{H}$ is a locally free sheaf by Lemma \ref{lem:Exextension}, it follows by Lemma \ref{lem:bothagree} that $\mathrm{\Psi}^{-1}(\mathcal{H}) = (\mathcal{H})_{X^k}$. Hence
	$$ (\mathrm{\Psi}^{-1} \circ \mathrm{\Phi})(\OO_X(e)) = \mathrm{\Psi}^{-1}(\mathcal{H})[-1] = (\mathcal{H})_{X^k}[-1]. $$
	Together with \eqref{eqn:vanishinge}, the exact triangle \eqref{eqn:finaltriangle} becomes the short exact sequence \eqref{eqn:Hdecomp}, which is the universal equivariant extension of $\OO_{X^k}$ by $\bigoplus_{i=1}^k q_i^\ast \OO_X(e)$ since $\delta$ is a collection of evaluation maps. Therefore its induced connecting map in the long exact sequence of cohomology groups
	$$ H^0 \left( H^1(\OO_X(e)) \otimes \OO_{X^k} \right)^{\mathfrak{S}_k} \longrightarrow H^1 \left( \bigoplus_{i=1}^k q_i^\ast \OO_X(e) \right)^{\mathfrak{S}_k} $$
	is naturally an isomorphism, which implies $H^0((\mathcal{H})_{X^k})^{\mathfrak{S}_k} = 0$.
\end{proof}

We will eventually prove the stability of $E_x$ with respect to some ample class $H\in \NS(\Xk)$. Similar to the previous section we have 
\begin{equation*}
	\NS(\Xk)=\ZZ e_k\oplus\ZZ f_k\oplus\ZZ\delta.
\end{equation*}
For any $l \in \NS(X)$ and any ample class $h \in \NS(X)$ we have the intersection numbers
\begin{align*}
	l_{k}h_{k}^{2k-1} &= \frac{(2k-1)!}{(k-1)!2^{k-1}}(lh)(h^2)^{k-1}, \\
	\delta h_{k}^{2k-1} &= 0
\end{align*}
by \cite[Lemma 1.10]{wandel}. Moreover, by Lemma \ref{lem:Exextension} and \cite[Lemma 1.5]{wandel} we also have
$$ c_1(E_x) = c_1(\mathcal{H}) = c_1 (\OO_X(e)^{[k]}) = e_k - \delta. $$
It follows by the above formulas that for any ample class $h \in \NS(X)$, we have
$$ c_1(E_x) h_k^{2k-1} = \frac{(2k-1)!}{(k-1)!2^{k-1}}(eh)(h^2)^{k-1}. $$
However, $\OO_X(e)^{[k]}$ is a subsheaf of $E_x$ with the same $c_1$. For $E_x$ to be $\mu_{h_k}$-stable, it is necessary to have $eh<0$ since $h^2>0$. An easy computation shows that this condition cannot be fulfilled by the class $h=e+(2k-1)f$ from Lemma \ref{ample}, so we cannot hope that $E_x$ is $\mu$-stable with respect to the class $h_k$ induced by this $h$. However, for the class $\widehat{h}=(2k)e + (2k-1)f$ from Lemma \ref{ample}, we do have
\begin{align*}
	e\widehat{h}&=(2k)e^2+(2k-1)ef\\
	&=-(4k^2)+(4k^2-1)=-1.
\end{align*}

Indeed, in the rest of this subsection we will prove that $E_x$ is $\mu$-stable with respect to $\widehat{h}_k$. We use the same notation as in Section \ref{stabil1} and also need the following formula: assume $F$ is a coherent sheaf on $X^k$ with $\mathfrak{S}_{k}$-invariant Chern class 
\begin{equation*}
	c_1(F) = \sum\limits_{i=1}^{k} q_i^\ast c
\end{equation*}
where $c \in \NS(X)$, then the intersection number 
\begin{equation*}
c_1(F) \widehat{h}_{X^{k}}^{2k-1} = a_k (c\cdot\widehat{h})(\widehat{h}^2)^{k-1}
\end{equation*}
where $a_k = \frac{k(2k-1)!}{2^{k-1}}$; see \cite[Lemma 1.10]{wandel}. The main result of this subsection is the following

\begin{prop}
	\label{prop:stable}
	$E_x$ is $\mu$-stable with respect to $\widehat{h}_k$.
\end{prop}

\begin{proof}
	Assume that $F$ is a reflexive subsheaf of $E_x$ of rank $1\leqslant r \leqslant 2k-2$. We need to show that $\mu_{\widehat{h}_{k}}(F) < \mu_{\widehat{h}_{k}}(E_x)$. By \cite[Lemma 1.2]{stapleton}, it suffices to check that
	\begin{equation*}
		\mu_{\widehat{h}_{X^{k}}}((F)_{X^{k}}) < \mu_{\widehat{h}_{X^k}}((E_x)_{X^{k}}),
	\end{equation*}
	where $(F)_{X^{k}}$ is an $\mathfrak{S}_{k}$-invariant subsheaf of $(E_x)_{X^{k}}$. 
	
We apply the functor $j_\ast(\overline{\sigma}_{k, \circ}^\ast( (-)_\circ))$ to \eqref{defwrong2}. Since the functor is left exact, together with \cite[Lemma 1.1]{stapleton} we obtain that
\begin{equation}
	\label{eqn:leftexact}
	0 \longrightarrow (\mathcal{H})_{X^k}  \longrightarrow (E_x)_{X^k} \longrightarrow (I_{S_x})_{X^k} \longrightarrow Q \longrightarrow 0,
\end{equation}
such that $\supp(Q) \subseteq \mathrm{\Delta}$, where $\mathrm{\Delta} = X^k \setminus X_\circ^k$ is the big diagonal. It is also clear that
	\begin{equation*}
		\overline{\sigma}_{k, \circ}^\ast( (I_{S_x})_\circ) = \left( \bigotimes\limits_{i=1}^{k} q_i^\ast I_x\right)  \Bigg|_{X^{k} \setminus \mathrm{\Delta}}.
	\end{equation*}
Since $\mathrm{\Delta}$ is of codimension $2$ in $X^{k}$, we have that $c_1((I_{S_x})_{X^{k}}) = 0$. It follows that
\begin{equation*}
	c_1((E_x)_{X^{k}}) = c_1((\mathcal{H})_{X^{k}}).
\end{equation*}
Moreover, we have by \eqref{eqn:Hdecomp} that
\begin{equation*}
	c_1((\mathcal{H})_{X^k})=\sum\limits_{i=1}^{k }q_i^\ast e.
\end{equation*}
Therefore 
\begin{align*}
	c_1((E_x)_{X^{k}}) \widehat{h}_{X^{k}}^{2k-1}&=c_1((\mathcal{H})_{X^{k}}) \widehat{h}_{X^{k}}^{2k-1}\\
	&= a_k (e \widehat{h}) (\widehat{h}^2)^{k-1}\\
	&= a_k  (-1) (\widehat{h}^2)^{k-1}.
\end{align*}
Since $(F)_{X^{k}}$ is $\mathfrak{S}_{k}$-invariant, we have $c_1((F)_{X^k}) = \sum\limits_{i=1}^{k} q_i^\ast c$ for some $c \in \NS(X)$, and 
\begin{equation*}
	c_1((F)_{X^{k}}) \widehat{h}^{2k-1}_{X^{k}} = a_k(c\cdot\widehat{h})(\widehat{h}^2)^{k-1}.
\end{equation*}
We have the following two cases:

If $c\cdot\widehat{h} \leqslant -1$, then we have
\begin{equation*}
	c_1((F)_{X^{k}}) \widehat{h}^{2k-1}_{X^{k}} \leqslant c_1((E_x)_{X^{k}}) \widehat{h}_{X^{k}}^{2k-1} < 0.
\end{equation*}
Since $\rk((F)_{X^{k}}) < \rk((E_x)_{X^{k}})$, it follows that 
\begin{equation*}
	\mu_{\widehat{h}_{X^{k}}}((F)_{X^{k}}) < \mu_{\widehat{h}_{X^{k}}}((E_x)_{X^{k}}).
\end{equation*}	

If $c\cdot\widehat{h} \geqslant 0$, then $c_1((F)_{X^k}) \widehat{h}^{2k-1}_{X^k} \geqslant 0$. 

We choose a (not necessarily $\mathfrak{S}_{k}$-invariant) non-zero $\mu_{\widehat{h}_{X^{k}}}$-stable reflexive subsheaf of maximal slope $F' \subseteq (F)_{X^{k}}$, then $\mu_{\widehat{h}_{X^{k}}}(F') \geqslant 0$. However $q_i^\ast \OO_X(e)$ is $\mu_{\widehat{h}_{X^{k}}}$-stable  for $i=1,\ldots, k$, and 
\begin{equation*}
	c_1(q_i^\ast \OO_X(e)) \widehat{h}_{X^{k}}^{2k-1} = a_k(e \widehat{h})(\widehat{h}^2)^{k-1} = a_k(-1)(\widehat{h}^2)^{k-1} < 0.
\end{equation*}
Hence the only map from $F'$ to $q_i^\ast \OO_X(e)$ is zero. 

By \eqref{eqn:leftexact} we obtain a morphism $F' \stackrel{\alpha}{\rightarrow} (I_{S_x})_{X^{k}}$. It is clear that $(I_{S_x})_{X^{k}}$ is torsion free, so it is a subsheaf of its double dual $(I_{S_x})_{X^{k}}^{\vee\vee}$. We also note that the restriction of $(I_{S_x})_{X^{k}}$ on $X^{k} \setminus (\mathrm{\Delta} \cup q_1^{-1}(\{x\}) \cup \cdots \cup q_k^{-1}(\{x\}))$ is the trivial line bundle, hence 
\begin{equation*}
	(I_{S_x})_{X^{k}}^{\vee\vee} = \OO_{X^{k}}.
\end{equation*}
Therefore we have 
\begin{equation*}
	F' \stackrel{\alpha}{\rightarrow} (I_{S_x})_{X^{k}} \hookrightarrow \OO_{X^{k}}.
\end{equation*}
	If $\alpha \neq 0$, then the composition of both maps is non-zero, hence the stability forces 
	\begin{equation*}
		\mu_{\widehat{h}_{X^{k}}}(F') = 0 = \mu_{\widehat{h}_{X^{k}}}(\OO_{X^{k}}).
	\end{equation*}
	Since $F'$ is reflexive, the composition must be the identity map. Since $(I_{S_x})_{X^{k}} \neq \OO_{X^{k}}$ this is a contradiction. It follows that $\alpha = 0$, which implies by \eqref{eqn:leftexact} that $F'$ is a subsheaf of $(\mathcal{H})_{X^{k}}$. By \eqref{eqn:Hdecomp} and the above discussion, we can furthermore conclude that $F'$ is isomorphic to a subsheaf of the trivial bundle $H^1(\OO_X(e)) \otimes \OO_{X^{k}}$. The stability forces again that
	\begin{equation*}
		\mu_{\widehat{h}_{X^{k}}}(F') = 0 = \mu_{\widehat{h}_{X^{k}}}(\OO_{X^{k}})
	\end{equation*}
	and $F' \cong \OO_{X^{k}}$. Moreover, since all global sections of the trivial bundle $H^1(\OO_X(e)) \otimes \OO_{X^{k}}$ in \eqref{eqn:Hdecomp} are invariant under the permutation of $\mathfrak{S}_{k}$, we conclude that $F'$ itself is also $\mathfrak{S}_{k}$-invariant, which gives a non-trivial $\mathfrak{S}_k$-invariant global section of $\mathcal{H}_{X^k}$. This contradicts Lemma \ref{lem:Hdecomp}, therefore $(E_x)_{X^{k}}$ cannot be destabilized by any $\mathfrak{S}_{k}$-invariant subsheaf, which concludes that $E_x$ is $\mu_{\widehat{h}_{k}}$-stable.
\end{proof}

\subsection{A smooth connected component}

In this subsection, we will interpret the universal sheaf $\mathcal{E}$ as a family of stable sheaves on $X^{[n]}$ whose base is a smooth connected component of the corresponding moduli space. We have shown above that all the wrong-way fibers $E_x$ of the family $\mathcal{E}$ are $\mu$-stable with respect to $\widehat{h}_k$. We follow the idea in Theorem \ref{prop:sameH1} to show their $\mu$-stability with respect to a certain ample class near $\widehat{h}_k$.

\begin{thm}\label{prop:sameH2}
	There exists some ample class $H \in \NS(X^{[k]})$ near $\widehat{h}_k$, such that $E_x$ is $\mu_H$-stable for all $x \in X$ simultaneously.
\end{thm}

\begin{proof}
	The same as in Theorem \ref{prop:sameH1}, the value of $c = \mu_\beta(E_x)$ is independent of the choice of $x \in X$. We still define 
	$$ S := \{ c_1(F) \mid F \subseteq E_x \text{ for some } x \in X \text{ such that } \mu_\beta(F) \geqslant c \}. $$
	The proof of the present result is literally the same as the proof of Theorem \ref{prop:sameH1}, except that the step which shows that $S$ is a finite set has to be modified.
	
	For this purpose we make a few auxiliary definitions. Let $E'_x = \mathcal{G}^\vee \oplus I_{S_x}$ for each $x \in X$. We also define the set 
	$$ S' := \{ c_1(F') \mid F' \subseteq E'_x \text{ for some } x \in X \text{ such that } \mu_\beta(F') \geqslant c \}. $$
	We claim that $S \subseteq S'$. 
	
	Indeed, by \eqref{defwrong2}, every subsheaf $F \subseteq E_x$ is an extension of some subsheaf $F_2 \subseteq I_{S_x}$ by another subsheaf $F_1 \subseteq \mathcal{G}^\vee$. It is then clear that $F'= F_1 \oplus F_2$ is a subsheaf of $E'_x$, and that $c_1(F) = c_1(F')$. If $F$ destabilizes $E_x$, then $F'$ also destabilizes $E'_x$, which means that every element of $S$ is also in $S'$, as desired.
	
	It remains to show that $S'$ is finite. In fact, since $E'_x \subseteq (\mathcal{G}^\vee \oplus \OO_{X^{[k]}})$ for all $x \in X$, we obtain that $S'$ is a subset of
	$$ T' := \{ c_1(F') \mid F' \subseteq (\mathcal{G}^\vee \oplus \OO_{X^{[k]}}) \text{ such that } \mu_\beta(F') \geqslant c \}, $$
	which is finite by \cite[Theorem 2.29]{greb16}, hence $S'$ is also finite, which further implies the finiteness of $S$. This concludes the proof.
\end{proof}

Let $H$ be an ample class that satisfies Theorem \ref{prop:sameH2}, and $\mathcal{M}$ the moduli space of $\mu_H$-stable sheaves on $X^{[k]}$ with the same numerical invariants as $E_x$. Then the universal family $\mathcal{E}$ defines a classifying morphism
\begin{equation}\label{eqn:class2}
f \colon X \longrightarrow \mathcal{M}, \quad x\longmapsto [E_x].
\end{equation}

Similar as Theorem \ref{thm:component1}, we obtain

\begin{thm}\label{thm:component2}
	The classifying morphism \eqref{eqn:class2} defined by the family $\mathcal{E}$ identifies $X$ with a smooth connected component of $\mathcal{M}$.
\end{thm}

\begin{proof}
	For any pair of points $x, y \in X$, since $\mathrm{\Theta}$ is an equivalence, we have
	$$ \Ext^\ast(F_x, F_y) \cong \Ext^\ast(\OO_x, \OO_y); $$
	moreover by Remark \ref{eqn:pnfunc} we have
	$$ \Ext^\ast(E_x, E_y) \cong \Ext^\ast(F_x, F_y) \otimes H^\ast(\PP^{k-1},\CC). $$
	It is clear that
	$$ \Ext_X^\ast(\OO_x, \OO_y) \cong \begin{cases} \mathrm{\Lambda}^\ast(T_{X,x}) \quad & \text{if } x=y \\ 0 \quad & \text{if } x \neq y. \end{cases} $$
	Combining the above computations we obtain
	\begin{alignat*}{3}
		&\Hom_{\X}(E_x, E_y) &&= 0 \qquad &&\text{for any } x, y \in X \text{ with } x \neq y \\
		\text{and} \quad &\Ext^1_{\X}(E_x, E_x) &&\cong T_{X,x} \quad\ &&\text{for any } x \in X.
	\end{alignat*}
	These imply that $f$ is injective on closed points and that $\dim(T_{[E_x]}\mathcal{M})=2$ for all $x\in X$. The claim then follows from an argument similar to the proof of Theorem \ref{thm:component1}.
\end{proof}

\begin{rem}
The stable vector bundles constructed in Theorem \ref{prop:sameH1} as well as Theorem \ref{prop:sameH2} are not tautological bundles as the rank of a tautological bundle is always divisible by $k$, but in our cases the ranks are $k+1$ and $2k-1$.
\end{rem}

\end{document}